\documentclass[preprint]{elsarticle}

\usepackage{mathtools}
\usepackage{amssymb}
\usepackage{hyperref}
\usepackage{imakeidx}
\usepackage{tikz}
\usepackage{manfnt}
\usetikzlibrary{%
  matrix,%
  calc,%
  arrows%
}
\newcommand{\F}{\mathbb{F}}

\newcommand{\N}{\mathbb{N}}

\renewcommand{\P}{\mathbb{P}}
\newcommand{\Q}{\mathbb{Q}}
\newcommand{\RR}{\mathbb{R}}
\newcommand{\R}{\mathbb{R}}

\newcommand{\Z}{\mathbb{Z}}
\newcommand{\comment}[1]{}

\makeatletter
\def\imod#1{\allowbreak\mkern10mu\left({\operator@font mod}\,\,#1\right)}
\makeatother

\begin{document}

\title{Uniform Harbourne-Huneke Bounds via Flat Extensions}
\author{Robert M. Walker}

\address{Department of Mathematics, University of Michigan, Ann Arbor, MI, 48109}
\ead{robmarsw@umich.edu}

\newtheorem{theorem*}{Theorem}
\newtheorem{lemma*}{Lemma}
\newtheorem{theorem}{Theorem}[section]
\newtheorem{exm}[theorem]{Example}
\newtheorem{prop}[theorem]{Proposition}
\newtheorem{corollary}[theorem]{Corollary}
\newtheorem{lemma}[theorem]{Lemma}
\newtheorem{remark}[theorem]{Remark}
\newproof{proof}{Proof}
\newtheorem{conjecture*}{Conjecture}
\newtheorem{conjecture}[theorem]{Conjecture}
\newtheorem{Claim*}{Claim}
\newtheorem{definition}[theorem]{Definition}
\newtheorem{definition*}{Definition}
\newtheorem{observation}[theorem]{Theorem}
\newdefinition{remark*}{Remark}
\newtheorem{tble}[theorem]{Table}

\parskip=10pt plus 2pt minus 2pt

\begin{abstract} Over an arbitrary field $\F$, Harbourne  \cite{Primer} conjectured that the symbolic power $I^{(N (r-1)+1)} \subseteq I^r$ for all $r>0$ and all homogeneous ideals $I$ in $S = \F [\P^N] = \F[x_0, \ldots, x_N]$. The conjecture has been disproven for select values of $N \ge 2$: first by Dumnicki, Szemberg, and Tutaj-Gasi\'{n}ska in characteristic zero \cite{DSTG01}, and then by Harbourne and Seceleanu in positive characteristic \cite{resurge2}. However, the ideal containments above do hold when, e.g., $I$ is a monomial ideal in $S$ \cite[Ex.~8.4.5]{Primer}. 

As a sequel to \cite{walker2016}, we present criteria for containments of type $I^{(N (r-1)+1)} \subseteq I^r$ for all $r>0$ and certain classes of ideals $I$ in a prodigious class of normal rings.  Of particular interest is a result for monomial primes in tensor products of affine semigroup rings. Indeed, we explain how to give effective multipliers $N$ in several cases including: the 
$D$-th Veronese subring of any polynomial ring $\F[x_1, \ldots, x_n]$ 
$(n \ge 1)$; and the extension ring $\F[x_1, \ldots, x_n, z]/(z^D - x_1 \cdots x_n)$ of $\F[x_1, \ldots, x_n]$.
 \end{abstract}
 
\maketitle

%\tableofcontents

\section{Introduction and Conventions for the Paper}

Over an arbitrary field $\F$, let $S = \F [\mathbb{P}^N] = \F[x_0, x_1, \ldots, x_N]$ be the standard $\N$-graded polynomial algebra. The groundbreaking work of Ein-Lazarsfeld-Smith and Hochster-Huneke \cite{ELS, HH1}  implies that the symbolic power $I^{(Nr)} \subseteq I^r$ for all graded ideals 
$0 \subsetneqq I \subsetneqq S$ and all integers $r > 0$. Using graded ideals of \textit{star configurations}  in $\P^N$, Bocci and Harbourne \cite{BH1} showed that in securing these containments one cannot replace $N$ by some integer $0 < C < N$, even asymptotically.  In particular, $I^{(4)} \subseteq I^2$ holds for all graded ideals in $\F[\P^2]$, and Huneke asked whether an improvement $I^{(3)} \subseteq I^2$ holds for any radical ideal $I$ defining a finite set of points in $\P^2$. 
%Using graded ideals of \textit{star configurations}  in $\P^N$, Bocci and Harbourne (\cite{BH1}) showed that in securing these containments one cannot replace $N$ by some integer $0 < C < N$. 
%As part of an active area of research into symbolic powers in $S$, \textbf{ideal containment problems} concern the determination of when select families of ideal containments hold for a graded ideal $0 \subsetneqq I \subsetneqq S$, e.g., those of the type $I^{(m)} \subseteq I^r$. 
Building on this, Harbourne proposed dropping the symbolic power from $Nr$ down to the \textbf{Harbourne-Huneke bound} $Nr - (N-1) = N (r-1)+1$ when $N \ge 2$ \cite[Conj. 8.4.2]{Primer}: i.e.,  %(the $N=1$ case is simply the Ein-Lazarsfeld-Smith containment)%the family 
\begin{equation}\label{Harbourne-Huneke bound 001}
I^{(N (r-1) +1)} \subseteq I^r \mbox{ for any graded ideal $0 \subsetneqq I \subsetneqq S$, all }r > 0, \mbox{ and all }N \ge 2.
\end{equation} 
There are several scenarios where these improved containments hold: for instance, they hold for all monomial ideals in $S$ over any field \cite[Ex.~8.4.5]{Primer}; see the recent ideal containment problem survey by Szemberg and Szpond \cite[Thm.~3.8]{SzemSzpo01}, as well as recent work of Grifo-Huneke \cite{GrifoHun00} in 2017. 

%However, as was noted in the intro to (\cite{Walker002}), \eqref{Harbourne-Huneke bound 001} is not kismet: 
However, Dumnicki, Szemberg, and Tutaj-Gasi\'{n}ska showed in characteristic zero \cite{DSTG01} that the containment $I^{(3)} \subseteq I^2$ can fail for a radical ideal defining a point configuration in $\P^2$.  
Harbourne-Seceleanu showed in odd positive characteristic \cite{resurge2} that \eqref{Harbourne-Huneke bound 001} can fail for pairs $(N, r) \neq (2, 2)$ and ideals $I$ defining a point configuration in $\P^N$. 
 Akesseh \cite{Akes001} cooks up many new counterexamples to  \eqref{Harbourne-Huneke bound 001}  from these original constructions via finite, flat morphisms $\varphi^\# \colon \P^N \to \P^N$. 
No prime ideal counterexample is known. 
%\noindent Tighter containments of the type $I^{(E(r-1)+1)} \subseteq I^r$, as in Theorems \ref{thm:ToricUSTPMonoPrimes001} and \ref{thm:VarPi-FSig-Bound01}  and first promoted by Harbourne, hold for  all monomial ideals in an affine polynomial ring over any field \cite[Ex.~8.4.5]{Primer}; see also recent work of Grifo-Huneke \cite{GrifoHun00}.  

Lately, there has been better sustained success in showing that a containment in \eqref{Harbourne-Huneke bound 001} fails--and perhaps, more fervor. However, we want to revisit the fact that in arbitrary characteristic \eqref{Harbourne-Huneke bound 001} holds for all monomial ideals in $S$. In particular, our investigation of Harbourne-Huneke bounds \textit{improves upon} the fact that $P^{(N (r-1) +1)} \subseteq P^{(r)} = P^r \mbox{ for all }r > 0$ and for all monomial prime ideals $P$ in $S$ (i.e., monomial ideals generated by %nonempty 
subsets of the variables $x_0, \ldots, x_N$). Indeed, $P^{(r)} = P^r$ for all $r$ whenever $P$ is a complete intersection ideal in $S$, and for all $r > 0$, $N (r-1) + 1 \ge (r-1) +1 = r$. 

The goal of this paper is to %present evidence 
show that a variant of  \eqref{Harbourne-Huneke bound 001}  %the Harbourne-Huneke bound %might hold 
holds for several familiar classes of ideals (e.g., %those with rich combinatorial structure
combinatorial ideals such as monomial primes) in %fairly nice
certain non-regular rings--even though it already fails for a large class of ideals defining point configurations in $\P^N$, and hence can fail for arbitrary graded ideals in %a polynomial ring 
$\F[\P^N]$. More precisely, %for a group-theoretic variant of the Harbourne-Huneke bound (e.g., for prime ideals) 
we work in the setting of rational surface singularities and higher-dimensional normal toric rings.
First, we demonstrate how one can strengthen Lemmas 1.1 and 2.6 of our IJM paper \cite{walker2016} to a version involving %modularity and 
a Harbourne-Huneke bound:

\begin{lemma}\label{thm: du Val bound 1}
\textit{Let $R$ be a Noetherian normal domain of positive Krull dimension whose global divisor class group $\operatorname{Cl}(R) : = \operatorname{Cl}(\operatorname{Spec}(R))$ is annihilated by an integer $D>0$. Then %the symbolic power 
$$\mathfrak{q}^{(D(r-1) + s)} = (\mathfrak{q}^{(D)})^{r-1} \mathfrak{q}^{(s)} \mbox{, and } \quad \mathfrak{q}^{(D(r-1) + 1)} \subseteq \mathfrak{q}^r$$ for all ideals $\mathfrak{q} \subseteq R$ of pure height one, all $r>0$, and all $0 \le s < D$.} 
\end{lemma}

\noindent %Of course, the proof we give allows Lemma 2.3 of (\cite{Walker001}) to be strengthened analogously. 
%In particular, when $D = 2$ works, we have $\mathfrak{q}^{(3)} \subseteq  \mathfrak{q}^2$ for all ideals $\mathfrak{q} \subseteq R$ of pure height one. 
When the domain $R$ in this lemma is two-dimensional,  $P^{(r)} = P^r$ when the ideal $P$ is zero or maximal, and so
 we infer that %$P^{(Da + r)} = (P^{(D)})^a P^{(r)}$,  and
 $P^{(D(r-1) + 1)} \subseteq P^r$ for all prime ideals $P$ in $R$ and  all $r> 0$, and that $P^{(3)} \subseteq  P^2$ for all primes when $D = 2$ works. As discussed in \cite{walker2016}, the above lemma already applies to any two-dimensional, local %or $\N$-graded local 
 rational singularity (Lipman \cite{Lip0}) and the coordinate rings of simplicial toric varieties; see Theorem \ref{thm: exact sequence 00} below. The intro to \cite{walker2016} gives Lipman's definition of two-dimensional, normal local rational singularities; Section 3 therein gives remarks on class groups, both for these singularities and for toric varieties. %See also (\cite{SinghSpiroff000}).
%For instance, the condition on the class group is met by the two-dimensional, normal local rational singularities Lipman studied in (\cite{Lip0}), as well as by the coordinate ring of any normal affine toric curve or surface.
 %two-dimensional normal affine toric variety.
 We prove a result for Veronese rings (Theorem \ref{thm: Veronese rings are optimal 1} below) from which one can infer that the ideal containment in the lemma can be tight by example. 
 %is \textbf{sharp}. 
%Analogous with the uniformity results of (\cite{ELS,HH1,HKV}), Lemma \eqref{thm: du Val bound 1} can be adduced in support of Huneke's philosophy in (\cite{hun1}) that there are uniform bounds lurking in commutative algebra. 

However, it is the result to follow that inspires the chosen title for this paper. It allows us to give first examples of 
 the Harbourne-Huneke bound for all 
 monomial primes in 
certain normal algebras of dimension three or higher, subalgebras of a Laurent polynomial ring that are generated by monomials. 
These domains are the coordinate rings of 
  normal affine toric varieties, called toric rings, monomial rings,  or affine semigroup rings. 
%To state our main results in a special case, fix a ground field $\F$. We fix a full-dimensional pointed convex polyhedral cone $C \subseteq \R^n$ generated by  a finite set $G \subseteq \Z^n$, and its dual $C^\vee \subseteq \RR^n$.   Let $R_\F$ be the semigroup algebra of the semigroup $C^\vee \cap \Z^n$. This  \textbf{toric $\F$-algebra associated to $C$} is a normal domain of finite type over $\F$ \cite[Thm.~1.3.5]{torictome} with an $\F$-basis of monomials $\{\chi^\ell \colon \ell \in C^\vee \cap \Z^n \}.$  An ideal of $R_\F$ is \textbf{monomial  (or torus-invariant)} if it is generated by a subset of these monomials. In what follows, we use $\bullet$ to denote dot products in $\RR^n$.  
In this setting, we adduce a result (Proposition \ref{prop: faithful flatness criterion 001}) on ideal containment preservation along faithfully flat ring extensions, as part of deducing the following  

\begin{theorem}\label{thm: finite tensor products 000}
\textit{Let $R_1, \ldots, R_n$ be normal 
affine semigroup rings over a field $\F$. % $P_i \subseteq R_i$ monomial primes with $1 \le  i \le n$. Let
For each $1 \le i \le n$,  suppose there is an integer $D_i>0$ such that $P^{(D_i (r-1) + 1)} \subseteq P^r$ %and $(P R)^{(D_i (r-1) + 1)} \subseteq (P R)^r$ 
for all $r> 0$ and all monomial primes $P \subseteq R_i$. Set $D := \max\{D_1, \ldots, D_n \}$. Then $Q^{(D (r-1) + 1)} \subseteq Q^r$ for all $r>0$ and any monomial prime $Q$ in the normal 
affine semigroup ring %normal toric ring
 $R = R_1 \otimes_\F \cdots \otimes_\F R_n$. 
%$D = D_1 + \cdots + D_n$. 
 % If in fact for each $1 \le i \le n$ there is an integer $D_i>0$ such that $P^{(N)} \subseteq P^{\lceil N/D_i\rceil}$ %and $(P R)^{(D_i (r-1) + 1)} \subseteq (P R)^r$ 
%for all $N> 0$ and all monomial primes $P \subseteq R_i$. Then $Q^{(N)} \subseteq Q^{\lceil N/ D \rceil}$ for all $N>0$, where $D = \max\{D_1, \ldots, D_n \}$.
}  
\end{theorem}
\noindent To clarify, a normal affine semigroup $\F$-algebra $A$ has an $\F$-basis of Laurent monomials and an ideal in $A$ is \textbf{monomial} if it is generated by monomials. See Section \ref{sec:Toric-Alg-Prelims00} for more details. 

All normal toric rings of dimension at most two have finite cyclic divisor class group, and thus satisfy the hypotheses on the $R_i$ factors in the theorem; %, with $D_i$ being the order of the class group;
 aside from these cases, 
 the factors $R_i$ may be taken from the following classes of rings (including those of Krull dimension three or higher): 
\begin{theorem}\label{thm: Veronese hypersurf}
\textit{Let $S = \F [x_1, \ldots, x_n]$ $(n \ge 1)$ be a polynomial ring over an arbitrary field $\F$ and consider the module-finite extensions of normal toric rings   
$V_D \subseteq S \subseteq H_D$, where %Suppose $R = \F [\sigma^\vee \cap \Z^n]$ is one of the following toric rings:
\begin{enumerate}
\item $V_D \subseteq S$ is the $D$-th Veronese subring with its standard $\N$-grading, and   
\item $H_D = \F[z, x_1, \ldots, x_n]/(z^D - x_1 \cdots x_n)$ is a  hypersurface ring. 
\end{enumerate} 
Then %there are $($at least$)$ $2^n$ \textbf{distinguished} %for which
 $P^{(D(r-1) + 1)} \subseteq P^r$ for all $r> 0$, where $P$ is a monomial ideal in  any of the three rings. %In the case of $R_D$ these containments are optimal when $P \neq 0$ is not the graded maximal ideal. Also, for $D = r = 2$, $P^{(3)} \subseteq P^2$.
} 
\end{theorem}

%The paper is organized as follows. In Section \eqref{section: Mars Walker Lemma}, we prove Lemma \eqref{thm: du Val bound 1} and Proposition \eqref{prop: faithful flatness criterion 001}.  Section \eqref{section: Main Event} reviews most of the relevant toric algebraic geometry, culminating in a proof of Theorem \eqref{thm: finite tensor products 000}.  
%In Section \eqref{section: Main Event 2}, we prove Theorem \eqref{thm: Veronese hypersurf}, % Section  \eqref{section: Group-Theoretic 1} clarifies the connection between Theorem \eqref{thm: Veronese hypersurf} 
%along with clarifying that theorem's connection with divisor class groups. %Section \eqref{section: Tensoring Veronese 1} shows by example how the conclusion of Theorem \eqref{thm: finite tensor products 000} can be improved for finite tensor products of Veronese rings. %one can deduce Harbourne-Huneke bounds on monomial primes in finite tensor products of these Veronese and $D$-th power extension rings.
% Section \eqref{section: Finale 1} is succinct, closing the paper with some lingering questions.

\noindent \textbf{Conventions:} All our rings are Noetherian and commutative with identity. %, and are (non-Artinian) Noetherian normal domains, unless stated otherwise.
From Section \ref{section: Main Event} onwards, our rings will be \textit{affine} $\F$-algebras, that is, of finite type over a fixed field $\F$ of arbitrary characteristic. %All our algebraic varieties are irreducible; by 
By \textit{algebraic variety}, we will mean an integral scheme of finite type over the field $\F$.  
%When we say $R$ is \textbf{graded local}, we mean that $R = \bigoplus_{d \ge 0} R_d$ can be graded by $\N$, with $R_0$ being a field, and $\mathfrak{m} = \bigoplus_{d > 0} R_d$ the unique homogeneous maximal ideal.  

\noindent \textbf{Acknowledgements:} This work forms part of my Ph.D. thesis. I thank my thesis adviser, Karen E. Smith, for encouraging me to write this manuscript, and for several fruitful discussions along the way. I also thank Daniel Hern\'{a}ndez, 
 Jack Jeffries, Luis N\'{u}\~{n}ez-Betancourt, and Felipe P\'{e}rez  for each critiquing a draft of the paper. I also thank an anonymous referee for comments that improved exposition. 
This work was supported by a NSF GRF under Grant Number PGF-031543, the NSF RTG grant 0943832, and a 2017 Ford Foundation Dissertation Fellowship.

\section{Symbolic Powers, Faithful Flatness, and the Proof of Lemma 1.1}\label{section: Mars Walker Lemma}

\subsection{Symbolic Powers and Faithful Flatness} 
%We present an a la carte sampling of Section 2 of (\cite{Walker001}).
 If $I$ is any proper ideal in a nonzero Noetherian ring $R$, and $\operatorname{Ass}_R (R/I)$ is the set of associated primes of $I$, we define its \textbf{$a$-th ($a \in \Z_{>0}$) symbolic power} ideal  $I^{(a)}$ by the rule: 
$$I^{(a)} := I^a W^{-1} R \cap R, \mbox{ where }W = R - \bigcup \{P \colon P \in \operatorname{Ass}_R (R/I)\}.$$ 
%where the multiplicative set $W = R - \bigcup \{P \colon P \in \operatorname{Ass}_R (R/I)\}$. 
Equivalently, $I^{(a)} = \{f \in R \colon sf \in I^a \mbox{ for some }s \in W\}.$ 
%In particular, if $I$ is radical with minimal primes $P_1, \ldots, P_t$, then $I^{(a)} = P_1^{(a)} \cap \ldots \cap P_t^{(a)}$. 
%Per (Atiyah-Macdonald  \cite{AtiyahMacdonald69}, Proposition 4.9), %and where $S := R - (\bigcup_{P \in Ass_R (R/I)} P)$ is the relevant multiplicative system, note that
% $I^{(1)} = I$ for any proper ideal $I \subseteq R$. In general, $I^{(a)} \supseteq I^a$ for all $a>1$.  
While $I^a \subseteq I^{(a)}$ for all $a$, the converse can fail for $a>1$: $I^{(1)} = I$ since $W$ is the set of nonzerodivisors modulo $I$.   
 %At present, we will be interested in symbolic powers of prime ideals and of ideals of \textbf{pure height one} (i.e., all associated primes have height one). 
 %In the next section, we implicitly use the fact that for any prime ideal $P$ in $R$, $P^{(N)} = P^N :_R (s)^\infty = \bigcup_{j \ge 0} (P^N :_R s^j)$ for any $s \not\in P$ belonging to all associated primes across all ordinary powers of $P$.  

Consider a flat map $\phi \colon A \to B$ of Noetherian rings. In what follows, the ideal $JB := \langle \phi(J) \rangle B$ for any ideal $J$ in $A$, and $J^r B = (JB)^r$ for all $r \ge 0$ since the two ideals share a generating set. 
 For any $A$-module $E$, the proof of Theorem 23.2 (ii) in Matsumura \cite{Matsumura} shows that 
\begin{equation}\label{eqn: associated primes and flatness 001}  
\operatorname{Ass}_B (E \otimes_A B) = \bigcup_{P \in \operatorname{Ass}_A (E)} \operatorname{Ass}_B (B/P B).
\end{equation}
We define a set 
$\mathcal{I}(A) = \{\mbox{proper ideals }I \subseteq A  \colon 
\operatorname{Ass}_{B} (B / IB) = \{ PB  \colon P \in \operatorname{Ass}_{A} (A / I) \} \}.$
Setting $E = A/I$ in \eqref{eqn: associated primes and flatness 001}, we observe that $I \in \mathcal{I}(A)$ if and only if the extended ideal $PB$ is prime for all $P\in \operatorname{Ass}_{A} (A/ I)$. Our paper \cite{Walker003} records a  simple example to illustrate that in an arbitrary faithfully flat ring extension, $\mathcal{I}(A)$ need not contain all prime ideals in $A$, let alone all proper ideals.   %\footnote{In the notation of this theorem in Matsumura's text, take the flat extension $A := R_i \subseteq B := R$, the $A$-module $E = R_i / I$ and the $B$-module $G = B = R$, so that $E \otimes_A  G  \cong R / I R$; if $P \in \mbox{Ass}_A (E)$ and $G/P G = R/PR $.} 
%Faithful flatness underpins the proof of the following 
%In particular, $I \in \mathcal{I}(R_i)$ if it satisfies the following condition:   if  $I = J_1 \cap \cdots \cap J_n$ is a primary decomposition in $R_i$, then $IR = J_1 R \cap \cdots \cap J_n R$ is a primary decomposition in $R$. This condition holds in a polynomial ring in finitely many variables over $R_i$, but can fail in an arbitrary faithfully flat extension ring (e.g., the extension $R = \frac{\RR[x]}{(x^2+1)} \hookrightarrow \C \otimes_{\RR} R = \frac{\C[x]}{(x^2+1)}$ is faithfully flat, the zero ideal in $R$ is prime, but it fails to extend to a primary ideal of $\frac{\C[x]}{(x^2+1)}$).  

\begin{prop}\label{prop: faithful flatness criterion 001}
Suppose $\phi \colon A \to B$ is a faithfully flat map of Noetherian rings. Then for each $I \in \mathcal{I}(A)$ and all integer pairs $(N, r) \in (\Z_{\ge 0})^2$, we have 
\begin{equation}\label{eqn: faithful flatness criterion 001}
I^{(N)} B = (IB)^{(N)},
\end{equation} 
and 
$I^{(N)} \subseteq I^r$ if and only if $(IB)^{(N)} =  I^{(N)}B \subseteq I^r B = (IB)^r. $ 
\end{prop}
\begin{proof}
First, $I^{(N)} B \subseteq (IB)^{(N)}$: indeed, if $f \in I^{(N)}$, then $s f \in I^N$ for some $s \in A$ such that 
$$s \not\in  \bigcup_{P \in \operatorname{Ass}_{A} (A/I)} P \stackrel{(\star)}{=} \bigcup_{P \in \operatorname{Ass}_{A} (A/I)} (PB \cap A) = \left(\bigcup_{P \in \operatorname{Ass}_{A} (A/I)} PB \right) \cap A  $$ where ($\star$) holds by faithful flatness; it follows that $s \not\in \bigcup_{P \in \operatorname{Ass}_{A} (A/I)} PB = \bigcup_{Q \in \operatorname{Ass}_{B} (B/IB) } Q$, where equality holds since $I \in \mathcal{I}(A)$ by hypothesis. We thus conclude that $f \in (IB)^{(N)}$. 

By definition, $(IB)^{(N)} B_W = (I B)^N B_W = I^N B_W$ since all three ideals contract to $(IB)^{(N)}$, where $B_W = W^{-1} B$ denotes the ring obtained via localization of $B$ at  the multiplicative system 
$$W = B- \left(\bigcup_{Q \in \operatorname{Ass}_{B} (B/IB) } Q\right) = B- \left(\bigcup_{P \in \operatorname{Ass}_{A} (A/I)} PB \right).$$ %again since $I \in \mathcal{I}(A)$.
 Notice that since $I^{(N)} B \subseteq (IB)^{(N)}$, the right-hand containment holds in
\begin{align*}
I^N B_W \subseteq I^{(N)} B_W = (I^{(N)}B) B_W \subseteq  (IB)^{(N)} B_W =  I^N B_W.
\end{align*}
Thus $I^{(N)} B$ and $(IB)^{(N)}$ localize to the same ideal $I^N B_W$; contracting back to $B$, we conclude that \eqref{eqn: faithful flatness criterion 001} holds for all $N \ge 0$. Finally, \eqref{eqn: faithful flatness criterion 001} gives both implications of the second part of the proposition, adducing faithful flatness once more to contract an ideal  containment to $A$.  \qed
\end{proof}
We adapt Proposition \ref{prop: faithful flatness criterion 001} later on (cf., Proposition \ref{prop: faithful flatness criterion 002}) to prove Theorem \ref{thm: finite tensor products 001}, from which Theorem \ref{thm: finite tensor products 000} follows as an immediate corollary.

\subsection{Preliminaries on Divisor Class Groups} 

Our main references are Fossum \cite{fossum2012divisor}, Hartshorne \cite[II.6]{Hartsh0}, Hochster \cite{hoch2}, and Matsumura \cite[Ch.~11]{Matsumura}. However, we opt to state mathematical definitions and results from these sources only for Noetherian normal domains, rather than for Krull domains in general as is done in \cite{fossum2012divisor}. 

%\noindent \textbf{Basics on Divisor Class Groups}

Throughout, $R$ will denote a Noetherian normal domain. Let $\mathcal{P}$ denote the set of height-one primes in $R$. 
As noted in Matsumura's chapter on Krull rings \cite[Corollary of Thm.~12.3]{Matsumura},  when $f \in R$ is a nonzero nonunit, and $\nu_P$ is the discrete valuation on the DVR $R_P$ (for $P \in \mathcal{P}$), we have a unique primary decomposition
$$(f)R = \bigcap_{P \in \mathcal{P}} P^{(N_P)}, \mbox{ where }N_P := \nu_P (f) = 0 \mbox{ for all but finitely many }P.$$
We define the \textbf{Weil divisor of $f$} to be $\mbox{div}(f) : = \sum_{P \in \mathcal{P} } N_P \cdot P$. Additionally, we define the \textit{trivial} effective Weil divisor $\operatorname{div}(\langle 1 \rangle R) = \operatorname{div}(R) =  [R] := 0$ of the unit ideal to have identically zero $\Z$-coefficients. 

\begin{definition} The \textbf{divisor class group} of a Noetherian normal domain $R$,  $$\operatorname{Cl}(R) = \operatorname{Cl}(\operatorname{Spec}(R)),$$ is the free abelian group on the set $\mathcal{P}$ of height one prime ideals of $R$ modulo relations $$a_1 P_1 + \ldots + a_r P_r = 0,$$ whenever the ideal $P_1^{(a_1)} \cap \ldots \cap 
P_r^{(a_r)}$ is principal. %(see Hochster's lecture notes \cite{hoch2} for more details on this definition; see also Hartshorne's presentation \cite[II.6]{Hartsh0}). 
%Additionally, we define the \textit{trivial} effective Weil divisor $\operatorname{div}(\langle 1 \rangle R) = \operatorname{div}(R) =  [R] := 0$ of the unit ideal to have identically zero $\Z$-coefficients. 
\end{definition}

In particular, $\operatorname{Cl}(R)$ is trivial if and only if $R$ is a UFD \cite[II.6]{Hartsh0}. Both conditions mean that every height one prime ideal in $R$ is principal. We note that $P^{(a)} = P^a$ for all $a>0$ and all height one primes $P$ in a Noetherian UFD.  

We now record three theorems without formal proof, consolidating some results from Ch.II, Sections 7, 8, and 10 of Fossum \cite{fossum2012divisor}. The first result consolidates some immediate consequences of a fact called Nagata's theorem \cite[Thm.~7.1]{fossum2012divisor}, on the behavior of class groups under localization.

\begin{theorem}[cf., Fossum {\cite[Cor.~7.2, Cor.~7.3]{fossum2012divisor}}]\label{thm:classgroupsunderlocalization}
Let $S$ be a multiplicatively closed subset of a Noetherian normal domain $A$. Then: 
\begin{enumerate}
\item The natural map $\operatorname{Cl}(A) \to \operatorname{Cl}(S^{-1} A)$ is a surjection of abelian groups. The kernel is generated by the classes of the height one prime ideals which meet $S$. 
\item If $S$ is generated by prime elements of $A$, then $\operatorname{Cl}(A) \to \operatorname{Cl}(S^{-1} A)$ is an isomorphism of abelian groups. 
\end{enumerate}
\end{theorem}

The next two results will be especially useful in  Section \ref{sec:Toric-Alg-Prelims00}. They allow us to reduce class group computations to particularly nice cases where we end up enjoying a more incisive handle on computing class groups up to isomorphism. 

\begin{theorem}[cf., Fossum {\cite[Thm.~8.1, Cor.~8.2]{fossum2012divisor}}]\label{thm:classgroupsunderpolynomialextensions}
Working with polynomial ring extensions of a Noetherian normal domain $A$, we have isomorphisms for any $n \in \Z_{>0}$: 
$$\operatorname{Cl}(A) \cong \operatorname{Cl}(A [X_1, \ldots, X_n]) \cong \operatorname{Cl}(A [X_1^{\pm 1}, \ldots, X_n^{\pm 1} ]).$$
\end{theorem}

\begin{proof}[Proof Sketch]
One can induce on $n$ with base case $n=1$. Assuming $n=1$, the left-hand isomorphism is the content of Fossum \cite[Thm.~8.1]{fossum2012divisor}. For the right-hand  isomorphism, apply Theorem \ref{thm:classgroupsunderlocalization}(2) to the polynomial ring $B = A[X]$ and the multiplicatively closed set $S \subseteq B$ generated by the prime element $X \in B$, so $S^{-1} B = A [X^{\pm 1}]$ is a Laurent polynomial ring in one variable over $A$. \qed   
\end{proof}

\begin{theorem}[cf., Fossum {\cite[Cor.~10.3, Cor.~10.7]{fossum2012divisor}}]\label{thm:classgroupsgradedbasechange}
Suppose that $A = \oplus_{n = 0}^\infty A_i$ is an $\N$-graded Noetherian normal domain where $A_0 = \F$ is a field, with homogeneous maximal ideal $\mathfrak{m} = \oplus_{n = 1}^\infty A_i$. Suppose that $\F'$ is any field extension of $A_0 = \F$, and that $A' := A \otimes_\F \F'$ is a Noetherian normal domain. Then $A'$ is faithfully flat over $A$ and the induced homomorphism 
$\operatorname{Cl}(A) \to \operatorname{Cl}(A')$ is injective. 
\end{theorem}

\subsection{The Proof of Theorem 1.1} 

We start by recalling the following proposition deduced in  \cite{walker2016}. To clarify, an ideal has \textbf{pure height $h$} is every associated prime has height $h$. In particular, none are embedded. 

\begin{prop}[cf., {\cite[Prop.~2.5]{walker2016}}]\label{prop: unmixed ideals 1}
Let $R$ be a Noetherian normal domain of positive Krull dimension, and $\mathfrak{q}$ any ideal of pure height one with associated primes $P_1, \ldots, P_c$. Then: 
\begin{enumerate}
\item[(a)] There exist positive integers $b_1, \ldots, b_c$, uniquely determined by $\mathfrak{q}$, such that the symbolic power $\mathfrak{q}^{(E)} = P_1^{(E b_1)} \cap \ldots \cap P_c^{(E b_c)}$ for all $E \ge 0$. 
\item[(b)] If either $(1)$ $D \cdot \operatorname{Cl}(R) = 0$, or $(2)$ the class $[\mathfrak{q}] \in \operatorname{Cl}(R)$ has finite order $D$, then for all integers $r \ge 0$, $\mathfrak{q}^{(Dr)} = (\mathfrak{q}^{(D)})^r$ is principal and  $\mathfrak{q}^{(Dr)} \subseteq \mathfrak{q}^r$.
\end{enumerate}
\end{prop} 

%This is proved as Proposition 2.2 of (\cite{Walker001}).
 Per part (a) of this proposition, we may define Weil divisors % Next, using (a)  we explain how $\mathfrak{q} \subseteq R$ as above gives an element in $\operatorname{Cl}(R)$. Taking $b_1, \ldots, b_r$ as in part (a), say 
%$\mathfrak{q}^{(E)} = P_1^{(E b_1)} \cap \ldots \cap P_r^{(E b_r)}$ for all $E > 0$ and some integers $b_1, \ldots, b_r>0$ uniquely determined by $\mathfrak{q}$. 
$$\mbox{div}[\mathfrak{q}] := b_1 \cdot P_1  + \cdots + b_c \cdot P_c, \quad  \mbox{div} [\mathfrak{q}^{(E)}] := E \cdot  \mbox{div} [\mathfrak{q}] = E b_1 \cdot P_1 + \cdots + E b_c \cdot P_c,$$ 
where $E >0$. In particular, $\mbox{div} [\mathfrak{q}^{(A+B)}] = \mbox{div} [\mathfrak{q}^{(A)}] + \mbox{div} [\mathfrak{q}^{(B)}]$ for all nonnegative integers $A$ and $B$. 

%\textbf{Proving the Lemma:} We are now ready for the  

\begin{proof}[Proof of Lemma \ref{thm: du Val bound 1}] Our proof of  the first claim replaces $r-1$ with $r \ge 0$. Per Proposition \ref{prop: unmixed ideals 1}(b), suppose 
$\mathfrak{q}^{(Dr)} = (\mathfrak{q}^{(D)})^r = (f^r)$  is principal for all $r \ge 0$ and some nonzero $f \in R$. Now set $I = \mathfrak{q}^{(s)}$. Following the first proof in Hochster's notes \cite{hoch2}, we have a short exact sequence 
$$0 \to \frac{(f^r)R}{(f^r)I} \to \frac{R}{(f^r)I} \to \frac{R}{(f^r) R} \to 0$$
and $\frac{(f^r)R}{(f^r)I} \cong R/I$ as $R$-modules via the $R$-linear map $\phi \colon R \twoheadrightarrow \frac{(f^r)R}{(f^r)I}$ with $\phi(g) = \overline{g  f^r} $. 
Thus per our exact sequence (cf., Matsumura \cite[Thm.~6.3]{Matsumura}), $$\varnothing \neq \mbox{Ass}_R (R/(f^r)I) \subseteq \mbox{Ass}_R (R/I) \cup \mbox{Ass}_R (R/(f^r)R)$$  and so $\mbox{Ass}_R (R/(f^r)I)$ contains only height one primes since the latter two sets do. Finally, comparing Weil divisors of pure height one ideals 
\begin{align*}
\mbox{div} [(f^r)I  =   (\mathfrak{q}^{(D)})^r \mathfrak{q}^{(s)}] &\stackrel{(*)}{=}  \mbox{div}[(f^r)R ] + \mbox{div}[I] \\
&= \mbox{div} [\mathfrak{q}^{(Dr)}] + \mbox{div} [\mathfrak{q}^{(s)}] = 
\mbox{div}  [\mathfrak{q}^{(Dr + s)}].
\end{align*}
As Hochster notes, one can check identity (*) after first localizing at each height one prime $Q$; in this case, the identity is obvious in a DVR.
%, as $\mbox{Cl}(S) = 0$ for any regular local ring $S$, being a UFD. 
Per (*), the two pure height one ideals $\mathfrak{q}^{(Dr + s)} ,  (\mathfrak{q}^{(D)})^r \mathfrak{q}^{(s)}$ have the exact same primary decomposition and hence are equal. To conclude:  since $\mathfrak{q}^{(D)} \subseteq \mathfrak{q}^{(1)} = \mathfrak{q}$, 
%dropping $r$ to $r-1$ and
 setting $s = 1$ yields 
$\mathfrak{q}^{(D(r-1) + 1)} = (\mathfrak{q}^{(D)})^{r-1} \mathfrak{q}^{(1)} \subseteq \mathfrak{q}^{r-1 + 1} = \mathfrak{q}^r$.  \qed\end{proof}

We close by remarking that after adapting the statement of \cite[Lem.~2.6]{walker2016} to feature the Harbourne-Huneke bounds, the exact same proof we gave in \cite{walker2016} will suffice. Namely, we reduce to the local case, and then invoke Lemma \ref{thm: du Val bound 1} from the present paper.

%\subsubsection{Symbolic Powers and Faithful Flatness}\label{section: Faithful Flatness Criteria} We close this section with a fact to be used near the end of the next section. 

\section{Toric Algebra Preliminaries}\label{sec:Toric-Alg-Prelims00}

We review notation and relevant facts from toric algebra, citing Cox-Little-Schenck \cite[Ch.1,3,4]{torictome} and Fulton \cite[Ch.1,3]{introtoric}. A lattice is a free abelian group of finite rank. We 
  fix a perfect bilinear pairing $\langle \cdot , \cdot \rangle \colon M \times N \to \Z$ between two lattices $M$ and $N$;  this identifies $M$ with $\operatorname{Hom}_\Z (N, \Z)$ and  $N$ with $\operatorname{Hom}_\Z (M, \Z)$.  
Our pairing extends to a perfect pairing of finite-dimensional vector spaces $\langle \cdot , \cdot \rangle \colon M_\RR \times N_\RR \to \RR$, where $M_\RR := M \otimes_\Z \RR$ and $N_\RR := N \otimes_\Z \RR$. % are dual. 

Fix an \textbf{$N$-rational} polyhedral cone and its $M$-rational dual: respectively, for some \textbf{finite} subset $G \subseteq N - \{0\}$ these are closed, convex sets of the form 
\begin{align*}
C &= \operatorname{Cone}(G) := \left\lbrace\sum_{v \in G} a_v \cdot v \colon \mbox{ each }a_v \in \RR_{\ge 0}\right\rbrace  \subseteq N_\R, \mbox{ and }   \\
C^\vee &:= \{w \in M_\R \colon \langle w , v \rangle \ge 0 \mbox{ for all }v \in C \} = \{w \in M_\R \colon \langle w , v \rangle \ge 0 \mbox{ for all }v \in G\}.
\end{align*}
By definition, the \textbf{dimension} of a cone in $M_\R$ or $N_\R$ is the dimension of the real vector subspace it spans; a cone is \textbf{full(-dimensional)} if it spans the full ambient space. A cone in $M_\R$ or $N_\R$ is \textbf{pointed (or strongly convex)} if it contains no line through the origin.  
A \textbf{face} of $C$ is a convex polyhedral cone $F$ in $N_\R$ obtained by intersecting $C$ with a hyperplane which is the kernel of a linear functional $m \in C^\vee$; $F$ is \textbf{proper} if $F \neq C$. 
When $C$ is both $N$-rational and pointed, so is every face $F$. Each such face $F \neq \{0\}$ has a uniquely-determined set $G_F$ of primitive generators. By definition, $v \in N$ is \textbf{primitive} if $\frac{1}{k} \cdot v \not\in N$ for all $k \in \Z_{>1}$.  

There is a bijective inclusion-reversing correspondence between faces $F$ of $C$ and faces $F^*$ of $C^\vee$, where $F^* = \{w \in C^\vee \colon \langle w , v \rangle = 0 \mbox{ for all }v \in F\}$ is the face of $C^\vee$ \textbf{dual to} $F$ \cite[Sec.~1.2]{introtoric}. Under this correspondence, 
either cone is pointed  if and only if the other is full, and   
\begin{equation}\label{eqn:face-duality-identity01}
\dim (F) + \dim (F^*) = \dim (N_\R) = \dim(M_\R).
\end{equation}

Fix an arbitrary ground field $\F$ and a cone $C$ as above in $N_\R$. The semigroup ring $R_\F = \F [C^\vee \cap M]$ is the \textbf{toric $\F$-algebra associated to $C$}. This ring $R_\F$ is a normal domain   of finite type over $\F$ \cite[Thm.~1.3.5]{torictome}. Note that $R_\F$ has an $\F$-basis $\{ \chi^m  \colon m \in  C^\vee \cap M \}$ of monomials, giving $R_\F$ an $M$-grading, where $\deg(\chi^{m}) := m$. 
A \textbf{monomial ideal (also called an $M$-homogeneous or torus-invariant ideal)} in $R_\F$ is an ideal generated by a subset of these monomials.  When $C^\vee$ is pointed, $R_\F$ also has a non-canonical $\N$-grading obtained by fixing any group homomorphism $M \to \Z$  taking positive values $C^\vee \cap M - \{0\}$. % on sends $C^\vee \cap M - \{0\}$ to $\Z_{>0}$. 
The set $\{\chi^m \colon m \in C^\vee \cap M - \{0\} \}$ generates the unique homogeneous maximal ideal $\mathfrak{m}$ under this $\N$-grading.  
 
\begin{remark*}\label{rem:strong-convexity-stipulation}
In forming the toric algebra $\F [C^\vee \cap M]$, there is no loss of generality in assuming $C$ is pointed in $N_\R$. Indeed, because $C^\vee \cap M = C^\vee \cap M'$ where $M' = M \cap \{\mbox{$\R$-span of $C^\vee$ in $M_\R$}\},$ we may replace $M$ by $M'$ to assume $C^\vee$ is full in $(M')_\R$. Now, replacing $N$ and $C$ by the duals of $M'$ and $C^\vee$, we may assume that $C$ is pointed in $N' = \operatorname{Hom}_\Z (M' , \Z)$. See \cite[Thm.~1.3.5]{torictome} for details. 
\end{remark*}

\noindent Fix a face $F$ of a pointed rational cone $C$:   \cite[p.53]{introtoric}   
 records a surjective $M$-graded ring map  
\begin{align*}
\phi_F \colon R_\F = \F[C^\vee \cap M] \twoheadrightarrow \F[F^* \cap M], \quad \phi_F (\chi^{m}) &= 
\begin{cases} \chi^{m}  &\mbox{ if $\langle m , v \rangle = 0$ for all }v \in F \\
0 &\mbox{ if $\langle m , v \rangle > 0$ for some }v \in F.   \end{cases} 
\end{align*}
Both rings are domains. The \textbf{monomial prime ideal of $F$}, $P_F : = \ker (\phi_F)$, has height equal to $\dim (F)$.     
Conversely, any monomial prime of $R_\F$ corresponds bijectively to a face of $C$. 

\begin{lemma}\label{lem:prime-positivity}
Fix a face $F$ of a pointed rational cone $C$, and the monomial prime $P_F \subseteq R_\F$ above. Let $G_F$ be the set of primitive generators of $F$, and set $v_F : =  \sum_{v \in G_F} v \in F \cap N$.  Then 
\begin{equation}\label{eqn:mono-prime-defn}
P_F = (\{\chi^{m} \colon \mbox{$m \in C^\vee \cap M$ and the integer $\langle m , v_F  \rangle > 0$}\})R_\F.
\end{equation}
\end{lemma}

\begin{proof}
First,  in defining $\phi_F (\chi^m)$ above, notice we can work with $v \in G_F$ without loss of generality. Now, fix $m \in C^\vee \cap M$. Then $\langle m , v \rangle \in \Z_{\ge 0}$ for all $v \in C \cap N$. As $\langle \cdot , \cdot   \rangle$ is bilinear, \eqref{eqn:mono-prime-defn} follows since a sum of nonnegative integers is positive if and only if one of the summands is positive.    
\end{proof}

\begin{prop}[Minkowski sum-Ideal sum]\label{prop:MinksumIdealsum}
Suppose $C \subseteq N_\RR$ is a pointed rational polyhedral cone, and $R_\F = \F [C^\vee \cap M]$ is the corresponding toric $\F$-algebra. When a face $F = \mbox{Cone}(G_F) = \rho_{1} + \ldots + \rho_{\ell}$ as a Minkowski sum of rays, 
\begin{equation}\label{eqn: Minkowski sum-ideal sum decomp 001}
P_F = \sum_{j=1}^\ell P_{\rho_{j}}
\end{equation} 
as a sum of ideals.
\end{prop} %\footnote{The converse for this sum decomposition is true in the simplicial case where Minkowski sum decomposion is unique.} 
\begin{proof}
Let $G_F = \{u_{\rho_j} \colon 1 \le  j \le \ell \}$ consist of the primitive ray generators. Any $v \in F$ satisfies $$v = \sum_{j=1}^\ell a_j u_{\rho_j} , \mbox{ for some }a_1, \ldots, a_\ell \in \RR_{\ge 0}.$$ 
Given any $w \in C^\vee$, $\langle w, v \rangle \ge 0$ for all $v \in C$. 
Thus for $v \in F$ as above, $$0 \le \langle w, v \rangle = \sum_{j=1}^\ell a_j \langle w, u_{\rho_j} \rangle, \mbox{ for some }a_1, \ldots, a_\ell \in \RR_{\ge 0},$$ and so $\langle w, v \rangle$ is positive if and only if $\langle w, u_{\rho_j} \rangle  > 0$ for some $1 \le j \le \ell$. We infer from this that the monomial ideals $P_F$ and $\sum_{j=1}^\ell P_{\rho_{j}}$ have a generating set in common, and hence are equal. \qed
\end{proof}
%Indeed, if $u_{\rho_j} \in \Z^n$ is the primitive ray generator of $\rho_j = \mbox{Cone}(u_{\rho_j})$, then $T = $$P_\tau $ (This is NOT an if and only if!)
%If $\sigma^\vee = \rho_1' + \ldots + \rho_n'$ is a simplicial cone of maximal dimension, the face $\rho_j^\perp \cap \sigma^\vee = \rho_j^*$ is a \textbf{facet} of $\sigma^\vee$, so there's a unique ray generator $u_{\rho_j'} \in \sigma^\vee \cap \Z^n$ such that the integer $D = u_{\rho_j'} \bullet u_{\rho_j} > 0$.    \textbf{We claim that} $P_{\rho_j}^{(D)} = \left(\chi^{u_{\rho_j'}}\right) R$.\footnote{This is certainly true in Krull dimension two, and in the case of Veronese rings, along with the dual hypersurface rings.} Taking this for granted, the ideal $I = \left(\{\chi^{u_{\rho_j'}} \colon 1 \le j \le n\}\right)$ is generated by a graded system of parameters, i.e., is $\mathfrak{m}$-primary where $\mathfrak{m} = P_{\sigma}$ is the unique $\Z^n$-graded maximal ideal: indeed, since $\sqrt{(\chi^{u_{\rho_j'}})R} = P_{\rho_j}$, $$\sqrt{I} = \sqrt{ \sqrt{(\chi^{u_{\rho_1'}})R} + \cdots + \sqrt{(\chi^{u_{\rho_n'}})R}} = \sqrt{P_{\rho_1} + \cdots + P_{\rho_n}} = \sqrt{\mathfrak{m}} = \mathfrak{m}.$$ 
%In the arguments to follow, we implicitly use the fact that for any prime ideal $P$ in $R$, $P^{(N)} = P^N :_R (s)^\infty = \bigcup_{j \ge 0} (P^N :_R s^j)$ for any $s \not\in P$ belonging to all embedded primes across all ordinary powers of $P$.

\begin{definition}
With notation as in Proposition \ref{prop:MinksumIdealsum}, we call \eqref{eqn: Minkowski sum-ideal sum decomp 001} a   \textbf{Minkowski sum-ideal sum decomposition} for $P_F$. 
\end{definition}

\begin{remark*} Adapting the proof of Proposition \ref{prop:MinksumIdealsum} accordingly, we could use any decomposition of $F$ as a Minkowski sum of faces, the latter need not be rays. 
\end{remark*}

Our next goalpost is Lemma \ref{lem: monomial primes as sums of extended ideals 001} on  decomposing monomial primes in tensor products of normal toric rings.  
%\subsection{Finite Tensor Products of Normal Toric Rings} 
Fix two pointed rational polyhedral cones  $C_i = \mbox{Cone}(S_i) \subset (N_i)_\RR$ $(i=1,2)$, where each $S_i$ consists of the primitive ray generators. Define lattices $N = N_1 \times N_2, M = M_1 \times M_2$ per the standing conventions. Let $\langle , \rangle_i \colon M_i \times N_i \to \Z$ and $\langle , \rangle \colon M \times N \to \Z$ indicate our three designated bilinear pairings. 
\begin{remark*}\label{rem:compatibilityofthreepairings}
While tedious, we could pedantically write down compatibility conditions to the effect that the output values of these pairings will agree relative to the obvious $\Z$-linear embeddings $N_i \hookrightarrow N$ and $M_i \hookrightarrow M$, e.g., $N_1 \cong N_1 \times \{0\}$. 
In particular, in a slight abuse of notation, going forward we identify $$\langle, \rangle = \langle , \rangle_1 + \langle, \rangle_2.$$ This generalizes the usual dot product setup naturally, $\Z^E \subseteq \R^E$, where $E = m+ n$ as a sum of positive integers.  
\end{remark*} 
The \textbf{product cone} $C = C_1 \times C_2$ in $N_\RR$ is a pointed rational polyhedral cone. In terms of ray generators, $C$ is generated as  
\begin{align*}
C = (C_1 \times \{0\})+ (\{0\} \times C_2) = \mbox{Cone}[(S_1 \times \{0\}) \cup (\{0\} \times S_2)]  \subseteq N_\RR.
\end{align*}
Note that $$C^\vee = (C_1 \times \{0\})^\vee \cap (\{0\} \times C_2)^\vee = C_1^\vee \times C_2^\vee.$$ For the right-hand equality, we defer to Remark \ref{rem:compatibilityofthreepairings}.  %, since a cone's dual is the intersection of closed half-spaces determined by its ray generators. 

%Our arguments can be adapted to any finite tensor product ($\otimes_\F$) of normal toric rings over $\F$. In particular, property \textbf{(1)} could be stated in a form involving a sum of $n$ monomial primes if we are tensoring $n$ normal toric rings. We summarize this discussion with a 

\begin{lemma}\label{lem: monomial primes as sums of extended ideals 001}
\textit{For $n \ge 2$, let $R_1, \ldots, R_n$ be normal toric rings over a field $\F$, built from pointed rational polyhedral cones $C_i \subseteq (N_i)_\RR$, respectively. Consider the normal toric ring $R \cong R_1 \otimes_\F \cdots \otimes_\F R_n$. Every monomial prime ideal $Q$ in $R$ can be expressed as a sum $Q = \sum_{i=1}^n (P_i R)$ of expanded  ideals, where each ideal $P_i  \subseteq R_i$ is a monomial prime.} % $Q = \sum_{j \in J} P_j R$ where $J \subseteq [n] = \{1, 2, \ldots, n\}$ and each ideal $P_j \neq 0 \subseteq R_j$ is a monomial prime.
\end{lemma}

\begin{proof}
Induce on $n$ with base case $n=2$; we focus on the base case for the remainder of the proof. Suppose $R_i = (R_i)_\F = \F [C_i^\vee \cap M_i]$, and  $$R = R_\F = \F [C^\vee \cap M] \cong R_1 \otimes_\F R_2.$$ %and every monomial prime $Q$ in $R$ decomposes as $Q = P' R + P'' R$ where $P'$ is a monomial prime in $R'$ and $P''$ is a monomial prime in $R''$. 
Any monomial prime in $R$ corresponds bijectively with a face of $C$. All faces of $C$ are of the form $F = F_1 \times F_2$ where $F_i$ is a face of $C_i$. Given $F$ as stated, with $Q_F \subseteq R$ the corresponding monomial prime, the base case follows from proving that 
\begin{enumerate}
\item[\textbf{(1)}] $Q_{F_1 \times F_2} = Q_{F_1 \times \{0\}} + Q_{\{0\} \times F_2} $; and 
\item[\textbf{(2)}] As expansions of monomial ideals, $Q_{F_1 \times \{0\}} = P_{F_1} R$,  
$Q_{\{0\} \times F_2} = P_{F_2} R$.
\end{enumerate}

The \textbf{Minkowski sum-ideal sum decomposition}  \eqref{eqn: Minkowski sum-ideal sum decomp 001} suffices to verify both claims. First, to  see  \textbf{(1)}, notice $F_1 \times F_2  = (F_1 \times \{0\}) + (\{0\} \times F_2)$ as a Minkowski sum of faces. 
As for \textbf{(2)}, \eqref{eqn: Minkowski sum-ideal sum decomp 001} allows us to reduce verification to the case where the $F_i$ are rays. We do so explicitly for $Q_{\rho \times \{0\}}$ where $\rho$ is a ray of $C_1$. 
We will use notations $\chi^a, \phi^b, \psi^c$  for characters in $R, R_1, R_2$ respectively. 
%Recall that as an algebra $R = \F [\chi^w \colon w \in \mathcal{H}_\sigma]$ where $\mathcal{H}_\sigma$ is the Hilbert basis of $S_\sigma$. Any $w \in \mathcal{H}_\sigma$ the Hilbert basis of $S_\sigma$ has the following \textbf{irreducibility property} (cf., discussion preceding Proposition 1.2.23 of \cite{torictome} Section 1.2): whenever $w = w' + w''$ with both $w', w'' \in S_\sigma$, then $w' = 0$ or $w'' = 0$. 
We express an arbitrary 
\begin{equation}\label{eqn:semigroup-rewrite}
w = (w_1, w_2) \in C^\vee \cap M  = (C_1^\vee \cap M_1) \times (C_2^\vee \cap M_2), 
%&= [(C_1 \times \{0\})^\vee \cap M] \cap [(\{0\} \times C_2)^\vee \cap M] \\
%&= S_{\sigma' \times \{0\}} \cap S_{\{0\} \times \sigma''},
\end{equation} where $w_i \in C_i^\vee \cap M_i$. 
For $w$ as in \eqref{eqn:semigroup-rewrite}, the three characters $\chi^w, \chi^{(w_1, 0)} = \phi^{w_1}, \chi^{(0,w_2)} = \psi^{w_2}$ all lie in $R$. Indeed, given any $v = (v_1, v_2) \in C$ with $v_i \in C_i$, and $w$ as in \eqref{eqn:semigroup-rewrite}, %\in S_\sigma = S_{\sigma' \times \{0\}} \cap S_{\{0\} \times \sigma''}$, 
all dot product terms below are \textit{nonnegative}: deferring to Remark \ref{rem:compatibilityofthreepairings}, 
\begin{align*}
\langle w, v \rangle &= \langle w_1, v_1 \rangle + \langle w_2 , v_2 \rangle \\
\langle (w_1, 0), v \rangle & =  \langle w_1, v_1 \rangle  \ge 0 , \quad 
\langle (0, w_2), v \rangle = \langle w_2 , v_2 \rangle \ge 0 .
\end{align*}    
%\textit{Any $w \in \mathcal{H}_\sigma$ the Hilbert basis of $S_\sigma$ has the following \textbf{irreducibility property:} whenever $w = w' + w''$ with both $w', w'' \in S_\sigma$, then $w' = 0$ or $w'' = 0$.}   
In particular, since $v \in C$ was arbitrary both  $(w_1, 0)$ and $(0, w_2)$ lie in $C^\vee \cap M$. 

Now suppose $\chi^{w} = \chi^{(w_1,0)} \chi^{(0,w_2)} = \phi^{w_1} \psi^{w_2} \in Q_{\rho \times \{0\}}$, i.e., $\langle w, v \rangle > 0$ for some vector $v = (v_1, v_2) \in \rho \times \{0\}$. Since $v_2 = 0$ here, equivalently $\langle w, v \rangle = \langle w_1, v_1 \rangle > 0$ for some $v_1 \in \rho$, i.e., the character 
$\chi^{(w_1,0)} = \phi^{w_1} \in P_{\rho} R$. Since $\chi^{(0, w_2)} = \psi^{w_2} \in R$, $\chi^{w} = \phi^{w_1} \psi^{w_2} \in  P_{\rho} R$. Thus $Q_{\rho \times \{0\}} \subseteq P_{\rho} R$. For the other inclusion: the characters $\chi^{(w_1,0)} = \phi^{w_1}$ as above generate $P_{\rho} R$, and each such generator lies in $Q_{\rho \times \{0\}}$ since we already indicated above that $\chi^w \in Q_{\rho \times \{0\}}$ if and only if $\chi^{(w_1,0)} = \phi^{w_1} \in P_{\rho} R$. \qed 
\end{proof}

\subsection{Hilbert Bases, Non-Full Cones, and Toric Divisor Theory}\label{subsub:Hilb00} 
First, suppose the pointed cone $C$ from Remark \ref{rem:strong-convexity-stipulation} is full.  Then there is a uniquely-determined minimal generating set $\mathcal{B}$ for $C^\vee \cap M$, in the sense that any other generating set contains $\mathcal{B}$. The set $\mathcal{B}$ is called the \textbf{Hilbert basis} of the semigroup, and consists of the \textbf{irreducible} vectors $m \in C^\vee \cap M - \{0\}$; a vector $v \in C^\vee \cap M$ is irreducible if it cannot be expressed as a sum of two vectors $m \in C^\vee \cap M - \{0\}$. 
See \cite[Prop.~1.2.17]{torictome} and \cite[Prop.~1.2.23]{torictome} for details.

 In case the pointed cone $C$ is not full, the next proposition is handy. 
\begin{prop}\label{prop: proof reduction}
Let $N'_\R$ be the $\R$-span of a pointed cone $C \subseteq N_\R$. Set $N' = N'_\R \cap N$, and consider $C$ as a full-dimensional cone in $N'_\R$ $($relabeled as $C' )$. Let $M' = \operatorname{Hom}_\Z (N', \Z)$ be the dual lattice. Then working over an arbitrary ground field $\F$, the toric ring $R_\F := \F [C^\vee \cap M]$ is isomorphic to $R_\F' \otimes_\F L$ where the toric ring $R_\F':= \F [(C')^\vee \cap M']$ and $L$ is a Laurent polynomial ring over $\F$. In particular, there is a bijective correspondence between the monomial primes of $R'_\F$ and $R_\F$ given by expansion and contraction of ideals along the faithfully flat ring map $\varphi \colon R'_\F \hookrightarrow R_\F' \otimes L = R_\F$. Moreover, the divisor class groups of $R_\F$ and $R_\F'$ are isomorphic. 
\end{prop}
\begin{proof}
While Cox-Little-Schenck \cite[Proof of Prop.~3.3.9]{torictome} yields the first assertion,  Lemma \ref{lem: monomial primes as sums of extended ideals 001} yields the second since a Laurent polynomial ring has no nonzero monomial primes. As for the class group assertion, $R_\F$ is a Laurent polynomial ring over $R_\F'$ after base change, so apply Theorem \ref{thm:classgroupsunderpolynomialextensions}. \qed 
\end{proof}

%Given a Noetherian normal domain $R$, the \textbf{divisor class group} $\operatorname{Cl}(R) = \operatorname{Cl}(\operatorname{Spec}(R))$ is the free abelian group on the set of height one prime ideals of $R$ modulo relations $ \sum_{i=1}^r a_i P_i = 0 $ when  the ideal $\bigcap_{i=1}^r P_i^{(a_i)}$ is principal.     Note $\operatorname{Cl}(R)$ is trivial if and only if $R$ is a UFD, i.e., all height one primes in $R$ are principal. 

We now recall how to compute divisor class groups up to isomorphism when working over algebraically closed fields. %the following   
%\begin{lemma} [Cf., Lem.~1.1 of \cite{Walker002}] \label{thm: du Val bound 1}   When every element of $\operatorname{Cl}(R) : = \operatorname{Cl(Spec}(R))$ is annihilated by an integer $D > 0$, written as $D \cdot \operatorname{Cl}(R) = 0$, the symbolic power $P^{(D(r-1)+1)} \subseteq P^r$ for all $r> 0$  and all prime ideals $P \subseteq R$ of height one. \end{lemma} 
Working over an algebraically closed field $\F$, fix a pointed cone $C$ as in Remark \ref{rem:strong-convexity-stipulation} and the pair of rings $R_\F$ and $R'_\F$ as in Proposition \ref{prop: proof reduction}. 
When $C \neq \{0\}$, each $\rho \in \Sigma(1)$, the collection of \textbf{rational rays (one-dimensional faces)} of $C$,  yields a unique primitive generator $u_\rho \in \rho \cap N$ for $C$ and a torus-invariant height one prime ideal $P_\rho$ in $R'_\F$; cf.,  \cite[Thm.~3.2.6]{torictome}.  
The torus-invariant height one primes generate a free abelian group $\bigoplus_{\rho \in \Sigma(1)} \Z  P_\rho$ which maps surjectively onto the divisor class group of $R'_\F$.  More precisely, we record the following well-known theorem; see \cite[Ch.~4]{torictome}.
\begin{theorem}\label{thm: exact sequence 00}
\textit{
With notation as in Proposition \ref{prop: proof reduction}, let $C \subseteq N_\R$ be a pointed cone with primitive  generators $\Sigma(1)$ as described above. %Let $C' = C \cap N'_\R$ be the corresponding pointed full-dimensional cone. 
Then there is a short exact sequence of abelian groups 
\begin{equation}\label{eqn:SES001}
0 \to M' \stackrel{\phi}{\to} \bigoplus_{\rho \in \Sigma(1)} \Z  P_\rho \to \operatorname{Cl}(R'_\F) \to 0,
\end{equation} 
where $\phi (m) = \operatorname{div}(\chi^m)=  \sum_{\rho \in \Sigma(1)}  \langle m , u_\rho\rangle  P_\rho$. %, and $\langle \cdot, \cdot \rangle$ is our bilinear pairing. 
Furthermore, $\operatorname{Cl}(R_\F)$ and $\operatorname{Cl}(R'_\F)$ are isomorphic,  $\operatorname{Cl}(R_\F)$ is finite abelian if and only if $C$ is simplicial, and trivial if and only if  $C$ is smooth. 
}
\end{theorem}

\begin{remark*}
This above result follows from \cite[Prop.~3.3.9, Prop.~4.1.1-4.1.2, Thm.~4.1.3, Exer.~4.1.1-4.1.2, Prop.~4.2.2, Prop.~4.2.6, and Prop.~4.2.7]{torictome}, essentially consolidating what facts we need to bear in mind going forward in the manuscript. 
\end{remark*}

\begin{definition}\label{defn:simplicalorsmoothcone}  The cone $C  \subseteq N_\R$ is \textbf{simplicial $($respectively, smooth$)$} if $C = \{0\}$ or the primitive ray generators form part of  an $\R$-basis for $N_\R$ $($resp., a $\Z$-basis for $N)$. We also apply the adjectives simplicial and smooth to the corresponding toric algebra $R_\F$ and the toric $\F$-variety $\operatorname{Spec}(R_\F)$.
\end{definition}

\begin{remark*}
In algebro-geometric language, if $C$ as in Theorem \ref{thm: exact sequence 00} is simplicial, then all Weil divisors on $\operatorname{Spec}(R_\F)$ are $\Q$-Cartier of index at most the order of $\operatorname{Cl}(R_\F)$.   
\end{remark*}

The next lemma says we can reduce all toric divisor class group computations to the case where $\F$ is algebraically closed, to leverage Theorem \ref{thm: exact sequence 00}. 
\begin{lemma}\label{lem:base-change-irrelevant}
\textit{With notation as in Proposition \ref{prop: proof reduction}, the divisor class groups %\ref{thm:height-one-primes}, 
$\operatorname{Cl}(R_\F) \cong \operatorname{Cl}(R_{\overline{\F}})$ are isomorphic.}
\end{lemma}

\begin{proof} By now it is clear we can reduce to the case where $C$ is a full pointed cone in $N_\RR$. The algebra $R_\F$ admits an $\N$-grading with its zeroth graded piece being $\F$; see the passage above Remark \ref{rem:strong-convexity-stipulation}. We may then cite Theorem \ref{thm:classgroupsgradedbasechange} %Fossum \cite[Cor.~10.5 on p.43]{fossum2012divisor} 
to conclude that up to isomorphism, $\operatorname{Cl}(R_\F) \subseteq \operatorname{Cl}(R_{\overline{\F}})$ as a subgroup. This improves to an equality for normal toric rings because the divisor classes of height one monomial primes belong to both groups and generate the latter by Theorem \ref{thm: exact sequence 00}. \qed 
\end{proof}

\begin{remark*}
Citing %Proposition \ref{prop: proof reduction} and 
Lemma \ref{lem:base-change-irrelevant}, we observe in passing that the two toric algebra results deduced in the IJM paper \cite[Thm.~1.3, Cor.~3.2]{walker2016} extend to the case where we are working %with non-full cones or working 
over arbitrary fields which need not be algebraically closed.   
\end{remark*}

\section{The Proof of Theorem 1.2: New Examples from Old}\label{section: Main Event}

Theorem \ref{thm: finite tensor products 000} is an immediate corollary, indeed a uniform bound analogue, of the following 
\begin{theorem}\label{thm: finite tensor products 001}
For $n \ge 2$, let $R_1, \ldots, R_n$ be normal toric rings over a field $\F$, $P_i \subseteq R_i$ monomial primes with $1 \le  i \le n$. For each $1 \le i \le n$,  suppose there is an integer $D_i>0$ such that $P_i^{(D_i (r-1) + 1)} \subseteq P_i^r$ %and $(P_i R)^{(D_i (r-1) + 1)} \subseteq (P_i R)^r$
 for all $r> 0$. Set $D = \max\{D_1 , \ldots, D_n\}$.  Then $Q^{(D (r-1) + 1)} \subseteq Q^r$ for all $r>0$, where the monomial prime $Q = \sum_{i = 1}^n (P_i R) \subseteq R \cong R_1 \otimes_\F \cdots \otimes_\F R_n$. %is  the corresponding monomial prime in the normal toric ring $R \cong R_1 \otimes_\F \cdots \otimes_\F R_n$.
\end{theorem}

Prior to giving the proof, we will state two preliminary lemmas, proving the latter lemma. %: in both statements $S$ \textbf{is a nonzero Noetherian ring}. 

\begin{lemma}[{\cite[Ch.~3]{5authorSymbolicSurvey}}]\label{lem: symb powers as saturations 001}
For any prime ideal $P$ in a Noetherian ring $S$, and $N \in \Z_{\ge 0 }$,  $$P^{(N)} = P^N :_S (s)^\infty = \bigcup_{j \ge 0} (P^N :_S (s^j)) = P^N :_S (s^T)$$ for all $T \gg 0$ and any $s \not\in P$ belonging to all embedded primes of $P^N$.
\end{lemma}

\begin{lemma}\label{lem: equivalence of symb power containments 001}
Given any proper ideal $I$ in a Noetherian ring $S$, and $E \in \Z_{\ge 0 }$, 
$$\mbox{$\operatorname{(1) }$   $I^{(N)} \subseteq I^{\lceil N / E \rceil}$ for all $N \ge 0$ } \iff \mbox{$\operatorname{(2) }$  $I^{(E (r-1) + 1)} \subseteq I^r$ for all $r>0$}.$$ 
\end{lemma}

\begin{proof}
The case $N=0$ is trivial (the unit ideal is contained in itself), so we show equivalence when $N>0$. Given $r>0$, setting $N = E (r-1) + 1$ in (1) gives (2). That (2) implies (1) follows from noticing that for any two positive integers  $N , r$, we have $r = \lceil N/ E \rceil$ if and only if $N = E (r-1) + j$ for some $1 \le j \le E$, and $I^{(m)} \subseteq I^{(n)}$ when $m \ge n$. \qed
\end{proof}

%\textbf{Remarks (!):} \textbf{(1)} In the arguments to follow, we implicitly use the fact that for any prime ideal $P$ in a Noetherian ring $S$, $P^{(N)} = P^N :_S (s)^\infty = \bigcup_{j \ge 0} (P^N :_S s^j) = (P^N :_S s^T)$ for all $T \gg 0$ and any $s \not\in P$ belonging to all embedded primes across all ordinary powers of $P$. \\
%\textbf{(2)} Also, given any proper ideal $I$ in $S$, the following are equivalent: (a) $I^{(N)} \subseteq I^{\lceil N / E \rceil}$ for all $N \ge 0$; and (b)  $I^{(E (r-1) + 1)} \subseteq I^r$ for all $r>0$. For $N>0$, setting $N = E (r-1) + 1$ in (a) gives (b). That (b) implies (a) follows from noticing that for any two positive integers  $N , r$, we have $r = \lceil N/ E \rceil$ if and only if $N = E (r-1) + j$ for some $1 \le j \le E$, and $I^{(m)} \subseteq I^{(n)}$ when $m \ge n$; the case $N=0$ is trivial (unit ideal). 

Finally, we adapt Proposition \ref{prop: faithful flatness criterion 001} to a specialized form suited to the proof. The backdrop will be as follows.   
Fix a field $\F$. For $n \ge 2$, fix nonzero $\F$-algebras $R_1, \ldots, R_n$. Since $R_2 \otimes_\F \cdots \otimes_\F R_n \neq 0$ is free and hence faithfully flat over $\F$, the tensor product $R = R_1 \otimes_\F R_2 \otimes_\F \cdots \otimes_\F R_n$ is faithfully flat over $R_1$; indeed, 
$R$ is faithfully flat over each $R_i$ by permuting the tensor factor under consideration (Cf., Exercise 9.11 in Altman-Kleiman \cite{Alt-Klei13}). Thus we can view the factors $R_i$ as subrings of $R$.

%For each $1 \le i \le n$, we define a set 
%\begin{align*}\mathcal{I}(R_i) = \{\mbox{proper ideals }I \subseteq R_i  \colon \operatorname{Ass}_{R} (R / IR) = \{ PR  \colon P \in \operatorname{Ass}_{R_i} (R_i / I) \} \}.\\\end{align*} 
%We will show that if $P$ is a monomial ideal in $R_i$, then $P^{(N)} R = (PR)^{(N)}$ for all $N>0$. Per the proof of Theorem 23.2(ii) in Matsumura (\cite{Matsumura}), it turns out that $I \in \mathcal{I}(R_i)$ if and only if the extended ideal $PR$ is prime for all $P\in \operatorname{Ass}_{R_i} (R_i / I)$.\footnote{In the notation of this theorem in Matsumura's text, take the flat extension $A := R_i \subseteq B := R$, the $A$-module $E = R_i / I$ and the $B$-module $G = B = R$, so that $E \otimes_A  G  \cong R / I R$; if $P \in \mbox{Ass}_A (E)$ and $G/P G = R/PR $.} In what follows, $\N = \Z_{\ge 0}$ is the set of nonnegative integers. Faithful flatness underpins the proof of the following 
%In particular, $I \in \mathcal{I}(R_i)$ if it satisfies the following condition:   if  $I = J_1 \cap \cdots \cap J_n$ is a primary decomposition in $R_i$, then $IR = J_1 R \cap \cdots \cap J_n R$ is a primary decomposition in $R$. This condition holds in a polynomial ring in finitely many variables over $R_i$, but can fail in an arbitrary faithfully flat extension ring (e.g., the extension $R = \frac{\RR[x]}{(x^2+1)} \hookrightarrow \C \otimes_{\RR} R = \frac{\C[x]}{(x^2+1)}$ is faithfully flat, the zero ideal in $R$ is prime, but it fails to extend to a primary ideal of $\frac{\C[x]}{(x^2+1)}$).  

\begin{prop}\label{prop: faithful flatness criterion 002}
Given the rings $R_i$ and $R$ as above, suppose that $R$ and each factor $R_i$ is Noetherian. Then for each $1 \le i \le n$, we have $I^{(N)} R = (IR)^{(N)}$ for all integers $N \ge 0$ where $$I \in \mathcal{I}(R_i) = \{\mbox{proper ideals }I \subseteq R_i  \colon \operatorname{Ass}_{R} (R / IR) = \{ PR  \colon P \in \operatorname{Ass}_{R_i} (R_i / I) \} \}.$$  Moreover, given $(N, r) \in (\Z_{\ge 0})^2$, 
$I^{(N)} \subseteq I^r$ if and only if $(IR)^{(N)} \subseteq (IR)^r.$ 
\end{prop}

\begin{proof}[Proof of Theorem \ref{thm: finite tensor products 001}]
By Propositions \ref{prop: proof reduction} and \ref{prop: faithful flatness criterion 002}, we may assume that all the toric rings $R_i$ and $R$ are built from full-dimensional pointed rational  polyhedral cones.  For each $1 \le i \le n$, let $x_{i, 1}, \ldots, x_{i, t_i}$ be the monomial algebra generators of $R_i$ over $\F$ corresponding to the Hilbert basis. By the isomorphism $R \cong R_1 \otimes_\F \cdots \otimes_\F R_n$, $R$ is the $\F$-algebra generated by $\{x_{i ,1}, \ldots, x_{i, t_i} \colon 1 \le i \le n \}$.  We define $S(N) := \{ (A_1, \ldots, A_n) \in (\Z_{\ge 0})^n \colon \sum_{i=1}^n A_i = N \}$ for each $N \ge 0$, so
\begin{align*}
Q^N = \left(\sum_{i = 1}^n (P_i R) \right)^N = \sum_{(A_1, \ldots, A_n) \in S(N)} \prod_{i=1}^n (P_i R)^{A_i}. 
\end{align*}  
We will in fact show that the monomial ideal 
\begin{align}\label{symbolic power partition containment 001}
Q^{(N)}  \subseteq \sum_{(A_1, \ldots, A_n) \in S(N)} \prod_{i=1}^n (P_i R)^{(A_i)}. 
\end{align}  
We may assume without loss of generality that all of the primes $P_i$ are nonzero. 

Take an arbitrary monomial $g = \prod_{i=1}^n m_i \in R$ where $m_i$ is a monomial in the $x_{i , \ell}$ for $1 \le i \le n$ and $1 \le \ell \le t_i$. Re-indexing if necessary, we may assume that 
$P_i R = (x_{i, 1}, \ldots, x_{i, s_i} )R$ for $1 \le i \le n$  
where $1 \le s_i \le t_i$, and this is a minimal generating set in the sense of Nakayama's Lemma since $R$ is $\N$-graded by the discussion preceding Remark \ref{rem:strong-convexity-stipulation}. Define a ``complement" monomial
$\mathcal{M} = \prod_{i=1}^n \prod_{\ell = s_i+1}^{t_i} x_{i , \ell}$ consisting of all algebra generators \textbf{not} among the generators of the $P_i R$. Any embedded prime of a power of $Q$ is graded (read, monomial), so that $Q^{(N)} = Q^N :_R (\mathcal{M})^\infty$ as a saturation per Lemma \ref{lem: symb powers as saturations 001}. If $g \in Q^{(N)}$, then for all $T \gg 0$, the monomial 
$$g \mathcal{M}^T \in Q^N =  \sum_{(A_1, \ldots, A_n) \in S(N)} \prod_{i=1}^n (P_i R)^{A_i},$$ whence for some $(A_1, \ldots, A_n) \in S(N)$ and each $1 \le j \le n$ we have $$g \mathcal{M}^T  =  \prod_{i=1}^n m_i \left(\prod_{\ell = s_i+1}^{t_i} x_{i , \ell} \right)^T 
= m_j  \prod_{i=1}^n m_i^{1 - \delta_{i j}} \left(\prod_{\ell = s_i+1}^{t_i} x_{i , \ell} \right)^T  \in  \prod_{j=1}^n (P_j R)^{A_j},$$ where $\delta_{ij}$ is the Kronecker delta.   %Since one could replace $\mathcal{M}$ in the saturation $Q^{(N)} = Q^N :_R (\mathcal{M})^\infty$ with any monomial in the $\{x_{i , s_i +1} \cdots x_{i, t_i} \colon 1 \le i \le n\}$ with \textbf{all exponents being positive},
We can express $(P_j R)^{(A_j)} = P_j R^{A_j} :_R (\mathcal{N})^\infty$ where $\mathcal{N}$ is any ``complement" monomial built from powers of all algebra generators \textbf{not} among the generators of $P_j R$. 
Therefore, setting $ \mathcal{N} = \mathcal{N}(j) =  \prod_{i=1}^n m_i^{1 - \delta_{i j}} \left(\prod_{\ell = s_i+1}^{t_i} x_{i , \ell} \right)^T$, % By flatness, we conclude that  $m_i \in  P_i^{(A_i)} R =  (P_i R)^{A_i} :_R (x_{i , s_i +1} \cdots x_{i, t_i} )R^\infty \subseteq (P_i R)^{(A_i)} $ 
 we see $m_j \in  (P_j R)^{(A_j)} $ for each $1 \le j \le n$. Thus $g  = \prod_{j=1}^n m_j \in \prod_{j=1}^n (P_j R)^{(A_j)}$. Since $g \in Q^{(N)}$ was arbitrary, \eqref{symbolic power partition containment 001} is immediate.

Finally, we show that \textbf{(!)} $Q^{(N)} \subseteq Q^{\lceil N / D \rceil}$ for all $N \ge 0$ where the integer $D = \max \{D_1, \ldots, D_n\}$. Using \eqref{symbolic power partition containment 001}: where $\mathcal{A} = (A_1 , \ldots, A_n) \in (\Z_{\ge 0})^n$,
\begin{align*}
Q^{(N)}  \stackrel{\eqref{symbolic power partition containment 001}}{\subseteq} \sum_{\mathcal{A} \in S(N)} \prod_{i=1}^n (P_i R)^{(A_i)} \stackrel{\mathbf{(A)}}{\subseteq} \sum_{\mathcal{A} \in S(N)} \prod_{i=1}^n (P_i R)^{ \lceil A_i  / D_i \rceil }  \stackrel{\mathbf{(B)}}{\subseteq} Q^{\lceil N / D \rceil}. 
\end{align*}  
To see \textbf{(A)}, by hypothesis, for all $1 \le i \le n$, $P_i^{(D_i (r-1) + 1)} \subseteq P_i^r$ for all $r>0$, so  $(P_i R)^{(D_i (r-1) + 1)} \subseteq (P_i R)^r$ for all $r>0$ by Proposition \ref{prop: faithful flatness criterion 002} since $P_i \in \mathcal{I}(R_i)$ by  the proof of Lemma \ref{lem: monomial primes as sums of extended ideals 001}. Thus by Lemma \ref{lem: equivalence of symb power containments 001},   %indeed, per \textbf{remark (!)} since $(P_i R)^{(D_i (r-1) + 1)} \subseteq (P_i R)^r$ for all $r> 0$ and $1 \le i \le n$, equivalently
 $(P_i R)^{(A_i)} \subseteq (P_i R)^{\lceil A_i / D_i \rceil}$ for all $A_ i \ge 0$ and all $1 \le i \le n$. 
 As for \textbf{(B)}: for each $(A_1, \ldots, A_n) \in S(N)$, we have  $\prod_{i=1}^n (P_i R)^{ \lceil A_i  / D_i \rceil }  \subseteq Q^{\lceil N / D \rceil}$; indeed, $\lceil A_i  / D_i \rceil  \ge \lceil A_i  / D \rceil$ for each $1 \le i \le n$, and the integer $\sum_{i=1}^n \lceil A_i  / D \rceil \ge \lceil (\sum_{i=1}^n A_i) / D \rceil  = \lceil N / D \rceil$ for each $(A_1, \ldots, A_n) \in S(N)$. To finish, since $Q^{(N)} \subseteq Q^{\lceil N / D \rceil}$ for all $N \ge 0$, we invoke Lemma \ref{lem: equivalence of symb power containments 001} again.    \qed
 %Setting $N = D(r-1)+1$ gives us the desired assertion.
\end{proof}

\section{Proving Theorem 1.3 in a Refined Form, Class Group Computations}\label{section: Main Event 2}
Theorem \ref{thm: Veronese hypersurf} is easy if $n = 1$ or $D=1$: all rings in sight are polynomial rings and monomial primes are complete intersections. Thus for the remainder of this section, \textbf{we will assume that $n \ge 2$ and $D \ge 2$.} 
We will give presentations of our rings as subrings of the domain of Laurent polynomials $L = \F[s_1^{\pm 1}, \ldots, s_{n-1}^{\pm 1}, u^{\pm 1}]$  in $n$ indeterminates over the field  $\F$.  The proof will proceed in cases, starting with the ring 
$H_D = \frac{\F [x_1, \ldots, x_n, z]}{(z^D - x_1 \cdots x_n) }.$

\begin{remark*}
Maintaining all notation conventions from Section \ref{sec:Toric-Alg-Prelims00}, in practice going forward we pick a basis $e_1, \ldots, e_n$ of our lattice $N$ with dual basis $e_1^*, \ldots, e_n^*$ for the dual lattice $M$, so that both $N$ and $M$ are isomorphic to $\Z^n$. Then the pairing $\langle \cdot , \cdot \rangle$ becomes the usual dot product.
\end{remark*}

%In the proofs of Theorems \ref{thm: Dth power hypersurface 1} and \ref{thm: Veronese rings are optimal 1} below, we implicitly use the fact that for any prime ideal $P$ in a Noetherian ring $R$, $P^{(E)} = P^E :_R (s)^\infty := \bigcup_{j \ge 0} (P^E :_R s^j)$ for any $s \not\in P$ belonging to all associated primes of $P^E$. %across all ordinary powers of $P$. 
% See \cite[Ch.~3]{5authorSymbolicSurvey} for details. 

\subsection{The Hypersurface Case:} We first observe that $H_D$ is a toric ring, up to isomorphism: 
\begin{lemma}\label{lem:Dth-power-hypersurf}
Consider the full-dimensional simplicial pointed rational polyhedral cone $\sigma_{D}^{(n)} \subseteq N_\RR \cong \RR^n$ whose ray generators are $D e_i + e_n \in N$ for $1 \le i < n$ and 
$e_n \in N$ in terms of the selected basis for $N$.
\begin{enumerate}
\item The Hilbert basis of the semigroup $(\sigma_D^{(n)})^\vee \cap M$ consists of $n+1$ vectors: the $n$ dual basis vectors $e_1^*, \ldots, e_n^*$, together with the vector $-e_1^* \cdots -e_{n-1}^* +  D e_n^* \in M$. 
\item The toric ring $\F[(\sigma_D^{(n)})^\vee \cap M]  \cong \frac{\F[x_1, \ldots, x_n , z]}{(z^{D} - x_1 \cdots x_n)} = H_D$.
\end{enumerate} 
\end{lemma}
\begin{proof}
The reader can use the hilbertBasis algorithm implemented in the \texttt{Polyhedra} package in Macaulay2 \cite{M2} to check (1). For (2), recall that to each $m = \sum_{i=1}^n m_i e_i^*  \in (\sigma_D^{(n)})^\vee \cap M$ we assign a Laurent monomial $\chi^m = s_1^{m_1} \cdots s_{n-1}^{m_{n-1}} u^{m_{n}}$ in the semigroup ring $\F[(\sigma_D^{(n)})^\vee \cap M]$. Given (1), in terms of $\F$-algebra generators we have %one can check  that the $n$-dimensional toric ring
$$\F [(\sigma_D^{(n)})^\vee \cap M] = \F \left[ s_1, \ldots, s_{n-1}, \frac{u^D}{(s_1 \cdots s_{n-1})} , u \right] \subseteq  \F[s_1^{\pm 1}, \ldots, s_{n-1}^{\pm 1}, u^{\pm 1}].$$ Given a polynomial ring $R = \F[x_1, \ldots, x_{n-1}, x_n, z]$ in $n+1$ variables, consider the surjective algebra map $\phi \colon R = \F[x_1, \ldots, x_{n-1}, x_n, z] \twoheadrightarrow \F [(\sigma_D^{(n)})^\vee \cap M]$ under which $x_i \mapsto s_i$ for each $1 \le i \le n-1$, $x_n \mapsto \frac{u^D}{(s_1 \cdots s_{n-1})}$, and $z \mapsto  u$. Since $\dim (R) = \dim(\F [(\sigma_D^{(n)})^\vee \cap M]) + 1$, we conclude that the kernel of $\phi$ is a height one prime in the UFD $R$, and hence is principal. Now $F = z^{D} - x_1 \cdots x_n \in R$ is irreducible by Eisenstein's Criterion  
and belongs to the kernel of $\phi$, so $\ker \phi = (F)$, and the isomorphism claim follows.  \qed 
\end{proof}

%In subsection \eqref{section: Group-Theoretic 1}, we show that 
%$\mbox{Cl} (H_D) \cong (\Z/D\Z)^{n-1}$ when $\F$ is algebraically closed. 

We now deduce the following refinement of Theorem \ref{thm: Veronese hypersurf} for $H_D$:
\begin{theorem}\label{thm: Dth power hypersurface 1}
Take the ring $H_D =  \F[x_1, \ldots, x_n , z]/(z^{D} - x_1 \cdots x_n),$ and $P$ one of the monomial prime ideals of $H_D$ $($i.e., $M$-graded /torus-invariant$)$; %(aside from the zero- or graded maximal ideal).
 assume $P$ is nonzero and nonmaximal. When $D \le \operatorname{ht}(P)$ $($the \textbf{height} of $P)$, $P^{(E)} = P^E$ for all $E >0$. If $D \ge \operatorname{ht}(P)$  and $E \equiv 1 \pmod{D}$, then  
$$P^{(E)} \subseteq P^{\operatorname{\operatorname{ht}}(P) \left(\frac{E-1}{D} \right) + 1}.$$
In particular, $P^{(Dr)} \subseteq P^{(D(r -1) + 1)} \subseteq P^{ \operatorname{\operatorname{ht}}(P) (r-1) +1} \subseteq P^r$ for all $r>0$. 
%and the middle containment is sharp.} %$($Also: $\mathfrak{q}^{(Dr - (D-1))} \subseteq  \mathfrak{q}^r$ for all $r>0$ when $\mathfrak{q}$ is any ideal of pure height one.$)$}
\end{theorem}
\begin{proof}
Citing the proof of Lemma \ref{lem:Dth-power-hypersurf}(2), the height $j$ prime ideal $P_j := (z, x_1, \ldots, x_{j})H_D$, for $1 \le j \le n-1$, equals $P_\tau$ for the $j$-dimensional face $\tau$ of 
$\sigma_{D}^{(n)}$ generated by $D e_i + e_n$ for $1 \le i \le j$. %and $e_n$.  
As a saturation, 
$P_j^{(E)} = P_j^E :_{H_D} (\prod_{i=j+1}^n x_i )^\infty$.
%\footnote{No matter what, the ideal $J = \bigcap_{Q \in A(P_j)-\{P_j\}} Q$ contains $\prod_{i=j+1}^n x_i $. \textbf{(Explain this footnote!)}} 
%Which monomials are in $P_j^{(N)}$: 
 Since $P_j^{(E)}$ is monomial, in chasing down inclusions below it suffices to discern which monomial classes
$g = (z^\ell x_1^{a_1} \cdots x_j^{a_j}) (x_{j+1}^{a_{j+1}} \cdots x_{n}^{a_n} ) \in H_D$ 
%\quad \quad \mbox{ (with $a_i >0$ for some $1 \le i \le j$)}$$
multiply a power of $m = \prod_{i=j+1}^n x_i $ into $P_j^E$. For $g$ as above, by definition $g \in P_j^{(E)}$ if and only if for all $T \gg 0$,
\begin{align*}
P_j^E \ni m^T g &= z^\ell \left(\prod_{i=j+1}^n x_i^{a_i + T} \right) 
\left(\prod_{i=1}^j x_i^{a_i}\right) \\
&= z^\ell \left(\prod_{i=1}^n x_i \right)^{T'}  \left( \prod_{i=j+1}^n x_i^{a_i + T- T'} \right) 
\left(\prod_{i=1}^j x_i^{a_i-T'}\right)\\
&= \left(z^{D \cdot T' + \ell} \prod_{i=1}^j x_i^{a_i-T'}\right)  \left( \prod_{i=j+1}^n x_i^{a_i + T- T'} \right),  
\end{align*}
where $T' = T'(T) := \min(a_1, \ldots, a_j, a_{j+1} + T, \ldots, a_n +T) = \min(a_1, \ldots, a_j)$ for all $T \gg 0$.  We conclude that 
$z^{D \cdot T' + \ell} \left(\prod_{i=1}^j x_i^{a_i-T'}\right)  \in P_j^E$, and infer the inequality 
%the multidegree of $h = x_{n+1}^{D \cdot M' + \ell}  \left(\prod_{i=1}^j x_i^{a_i-M'}\right)$ satisfies 
\begin{equation}\label{eqn: inequality 001}
(D-j) T' + \left( \sum_{i=1}^j a_i \right) + \ell  \ge E.
\end{equation}

Before proceeding, notice that since $T' \ge 0$, when $D \le j$ so that the number $(D-j)T'$ is nonpositive, \eqref{eqn: inequality 001} implies that $\left( \sum_{i=1}^j a_i \right) + \ell  \ge E$, so $(z^\ell x_1^{a_1} \cdots x_j^{a_j}) \in P_j^E$ and hence $g \in P_j^E$ already. Thus $P_j^{(E)} = P_j^E$ for all $E>0$ when $D \le j$, since both are generated by monomial classes. 
Thus in the remainder of the proof \textbf{we will assume that} $D \ge j = \operatorname{ht}(P_j)$,   
%Setting $N = D r - (D-1) = D(r-1) +1$ for an integer $r>0$, we show that $E := \ell + \sum_{i=1}^j a_i \ge r$, since this will imply that $z^\ell x_1^{a_1} \cdots x_j^{a_j} \in P_j^r$ and hence $g \in P_j^r$ when $g \in P_j^{(D(r-1)+1)}$. Suppose to the contrary that $E \le r-1$. Per \eqref{eqn: inequality 001}, $E + (D-j) M' \ge D(r -1) + 1$. Since $E \ge M' \ge 1$, and $D-1 \ge D-j$ (as $j \ge 1$),  
%\begin{align*}
%D(r-1) \ge D E = E + (D-1) E \ge E + (D-j) M' \ge D(r-1) +1
%\end{align*}
%a contradiction. Thus $E \ge r$ as desired. Ergo, $g \in P_j^r$ and since we took an arbitrary monomial $g \in P_j^{(D(r-1) +1)}$, we conclude that  $P_j^{(D(r-1) +1)} \subseteq P_j^r$ for all $r>0$ and all $P_j$. 
 i.e., $D-j \ge 0$. 
 
 In this case, assuming $E \equiv 1 \pmod{D}$, we now  show that %$P_j^{(N)} \subset P_j^b \iff b \le 1 + j \left(\frac{N-1}{D}\right)$. 
$P_j^{(E)} \subseteq P_j^{1 + j \left(\frac{E-1}{D}\right)}$. %The forward implication follows since for instance $(x_1 \cdots x_j)^{\left(\frac{N-1}{D}\right)} x_{n+1} \in     P_j^{(N)} \cap P_j^{1+ j\left(\frac{N-1}{D}\right)}$ as $$z^{\left(\frac{N-1}{D}\right)} (x_1 \cdots x_j)^{\left(\frac{N-1}{D}\right)} x_{n+1} = (x_{n+1}^D)^{\left(\frac{N-1}{D}\right)}  x_{n+1} = x_{n+1}^{N} \in P_j^{N}.$$ Note that this direction holds even if $D< j$. For the converse, it suffices to show containment when $b = 1 + j \left(\frac{N-1}{D}\right)$.
 Fix a monomial $$g = \left(z^\ell \prod_{i=1}^j x_i^{a_i} \right) \left(\prod_{i=j+1}^n x_{i}^{a_i} \right) \in P_j^{(E)},$$ and $T' = \min(a_1, \ldots, a_j)$ exactly as before. Now $g \in P_j^G$ where $G := \ell + \sum_{i=1}^j a_i$. The more involved case for us is when %We can assume without loss of generality  that
  (**) $T' \le (E-1)/D$: otherwise $$G \ge a_1 + \cdots+  a_j  \ge j T' \ge j (E-1)/D +  1,$$ whence one easily infers that $g \in P_j^{j \left(\frac{E-1}{D}\right) +1}$. Assuming (**), we now show that $G \ge j \left(\frac{E-1}{D}\right)+1$. Suppose to the contrary that 
$G \le j \left(\frac{E-1}{D}\right)$. Since $g \in P_j^{(E)}$, inequality \eqref{eqn: inequality 001} above says $$(D-j) T' + G  =  (D-j) T' + \left( \sum_{i=1}^j a_i \right) + \ell  \ge E  \Longrightarrow G  \ge E - (D-j)T'.$$ 
Then since %$E \ge M'  \ge 0$, $j \ge 1$, 
$E-1 - D T' \ge 0$ by (**), and $D-j \ge 0$, we see that 
\begin{align*}
j \left(E-1\right) = D j \left(\frac{E-1}{D}\right) \ge D G &\ge D E -  D (D-j)T' \\
&= D (E-1) + D  -  D (D-j)T'\\
&= j (E-1) + D + (D-j) (E-1- D T')\\
&\ge j (E-1) + D + (D-j)(0)\\
&= j(E-1) + D
\end{align*}
a contradiction. Thus $G \ge j \left(\frac{E-1}{D}\right) +1$, so $g \in P_j^{1+ j \left(\frac{E-1}{D}\right)}$. %Thus when $D \ge j$ and  we set 
In particular, when $E = D(r-1)+1$, we have
$P_j^{(D(r-1)+1)} \subseteq	 P_j^{1 + j (r-1)}.$ 
%In particular when $j =1$, $P_1^{(D(r-1)+1)} \subset P_1^b \iff b \le r.$ Thus our variant of a \textbf{Harbourne-Huneke containment} is sharp only in height one for this toric ring.
Finally, applying coordinate changes according to every permutation of $x_{[n]} := \{x_1, \ldots, x_n\}$, any (nonzero, nonmaximal) monomial prime ideal in $H_D$ can be obtained from the $P_j$ running through all indices $1 \le j \le n-1$, along with obtaining the desired containments.    \qed \end{proof}

\subsection{The Veronese Case:} Let $\N = \Z_{\ge 0}$ denote the set of \textbf{nonnegative integers}. To start, 
\begin{lemma}\label{lem:Vero-stuff}
Consider the full-dimensional simplicial  pointed rational polyhedral cone $\eta_{D}^{(n)} \subseteq N_\RR \cong \RR^n$ whose ray generators are $e_i$ for $1 \le i < n$ along with the vector  
$-e_1  -   \ldots -e_{n-1} + D e_n$ in terms of the basis selected for $N$. 
\begin{enumerate}
\item The Hilbert basis of the semigroup $(\eta_D^{(n)})^\vee \cap M$  is the set of vectors $$\left\lbrace e_n^*  + \sum_{i=1}^{n-1} a_i e_i^* \in M \colon \mbox{all $a_i \ge 0$ and } 0 \le \sum_{i=1}^{n-1} a_i  \le D \right\rbrace .$$ 
%it consists of  $\binom{D+(n-1)}{D}$ vectors.   
\item The toric ring $\F[(\eta_D^{(n)})^\vee \cap M]  \cong V_D$, the $D$-th Veronese subring of the polynomial ring $\F[s_1, \ldots, s_{n-1}, u]$ in the $n$ indeterminates $s_1, \ldots, s_{n-1}, u$.
\end{enumerate} 
\end{lemma} 
\begin{proof}
The reader can use the hilbertBasis algorithm implemented in the \texttt{Polyhedra} package in Macaulay2 \cite{M2} to check (1). Given (1), as an algebra over $\F$, we have
\begin{align*}
\F [(\eta_D^{(n)})^\vee \cap M] &= \F \left[s_1^{a_1} \cdots s_{n-1}^{a_{n-1}} u \colon \mbox{each $a_i \ge 0$, }0 \le \sum_{i=1}^{n-1} a_i  \le D \right]  
\\&\cong \frac{\F[x_{(a_1, \ldots, a_{n-1})} \colon \mbox{each $a_i \ge 0$, } 0 \le \sum_{i=1}^{n-1} a_i  \le D ]}{(x_e x_f - x_g x_h \colon e+f = g+ h \in \N^{n-1})}.
\end{align*}
Within the polynomial ring $\F[s_1, \ldots, s_{n-1}, u]$, applying the correspondence $$s_1^{a_1} \cdots s_{n-1}^{a_{n-1}} u \longleftrightarrow s_1^{a_1} \cdots s_{n-1}^{a_{n-1}} u^{D - a_1 - \cdots - a_{n-1}}$$  takes the generators in the presentation of $\F [(\eta_D^{(n)})^\vee \cap M]$ and recovers the usual presentation of $V_D$ in terms of degree 
$D$ monomials in $n$ variables. Therefore, (2) holds: $\F [(\eta_D^{(n)})^\vee \cap M] \cong V_D$.   \qed 
\end{proof}

%Using the hilbertBasis command in the Polyhedra package in Macaulay2, one can check that the $n$-dimensional toric ring is the degree-$D$ Veronese subring of the polynomial ring in $n$ variables over $\F$ (here, $\N = \Z_{\ge 0}$ denotes the set of \textbf{nonnegative integers}),
%\footnote{The correspondence $s_1^{a_1} \cdots s_{n-1}^{a_{n-1}} u \longleftrightarrow s_1^{a_1} \cdots s_{n-1}^{a_{n-1}} u^{D - a_1 - \cdots - a_{n-1}}$ recovers the usual presentation of $V_D$ in terms of degree $D$ monomials in $n$ variables.} 
%We show in subsection \eqref{section: Group-Theoretic 1} that $\mbox{Cl} (V_D) \cong \Z/D.$ 

We use the toric presentation of $V_D$ to deduce the following refinement of Theorem \ref{thm: Veronese hypersurf} for $V_D$:
\begin{theorem}\label{thm: Veronese rings are optimal 1}
\textit{Over an arbitrary field $\F$, take the $D$-th Veronese subring $V_D \subseteq \F[s_1, \ldots, s_{n-1}, u]$ and $P$ one of the monomial prime ideals of $V_D$.
When $P$ is nonzero and nonmaximal, $P^{(E)} \subseteq P^r$ if and only if $r \le \lceil E/D \rceil.$ In particular, $P^{(Dr)} \subseteq P^{(D(r -1) + 1)} \subseteq P^r$ for all $r>0$   
and the right-hand containment is sharp.} %$($Also: $\mathfrak{q}^{(Dr - (D-1))} \subseteq  \mathfrak{q}^r$ for all $r>0$ when $\mathfrak{q}$ is any ideal of pure height one.$)$}
\end{theorem}
\begin{proof}
Picking up from Lemma \ref{lem:Vero-stuff}, for all $1 \le j \le n-1$, we define height one primes  
$$P_j  = P_{e_j} = \left(s_1^{a_1} \cdots s_{n-1}^{a_{n-1}} u \colon a_j > 0 , \mbox{ and }1 \le \sum_{b=1}^{n-1} a_b \le D\right)V_D.$$  Then by the \textbf{Minkowski sum-ideal sum decomposition}  \eqref{eqn: Minkowski sum-ideal sum decomp 001} $P_{j_1 < \cdots < j_k} := P_{j_1}+\cdots + P_{j_k}$ is a prime of height $1 \le k \le n-1$ for each size-$k$ subset $j_1 < \ldots < j_k$ of $[n-1] = \{1, \ldots, n-1\}$.  In particular, we focus on $P_{1< \cdots < k} = (s^{\bar{a}} u \colon \bar{a} \in T_k)V_D$, where $$T_k :=\left\lbrace \bar{a} = (a_1, \ldots, a_{n-1}) \in \N^{n-1} \colon 
1 \le \sum_{b=1}^k a_b  \le \sum_{b=1}^{n-1} a_b  \le D \right\rbrace.$$ 
Any monomial $g$ in $P_{1 < \cdots < k}^{(E)} \subseteq P_{1 < \cdots < k} \subseteq P_{1 < \cdots < {n-1}}$ belongs to  $P_{1 < \cdots < k}$ and so decomposes (for some $B \ge 0$) as  
$$g = u^B \prod_{\bar{a} \in T_{n-1}} (s^{\bar{a}} u)^{i_{\bar{a}}} =   \prod_{\bar{a} \in T_k} (s^{\bar{a}} u)^{i_{\bar{a}}}  \left(u^B \prod_{\bar{a} \in T_{n-1} - T_k } (s^{\bar{a}} u)^{i_{\bar{a}}} \right) \in P_{1 < \cdots < k}^{\sum_{\bar{a}\in T_k} i_{\bar{a}} } .$$ %such that at most one $\N^{n-1}$-tuple satisfying $1 \le \sum_{b = k+1}^{n-1} a_b \le D$ appears.\footnote{If more than one appears we can iteratively apply Veronese relations/re-weighting until only one appears.} 
Note that this factorization of $g$ into two monomial pieces ($T_k$ versus $T_{n-1} - T_k$) is unique up to applying the Veronese relations $s^{\bar{e}} u \cdot s^{\bar{f}} u = s^{\bar{g}} u  \cdot s^{\bar{h}} u \quad (\bar{e} + \bar{f} =  \bar{g} + \bar {h}).$ 
%Now consider $R$ with its \textbf{standard $\N$-grading}. 
Setting the monomial $m := u \cdot  \prod_{\bar{a} \in T_{n-1} - T_{k}} s^{\bar{a}} u \in V_D$ to be the \textbf{product} of the monomials $s_1^{a_1} \cdots s_{n-1}^{a_{n-1}} u$ with $a_j = 0$ for all $1 \le j \le k (\le n-1)$, we have 
$P_{1< \cdots < k}^{(E)} = P_{1< \cdots < k}^E :_{V_D} (m)^\infty,$ and the monomial $g$ is in $P_{1< \cdots < k}^{(E)}$ precisely when for all $T \gg 0$, 
$$g \cdot m^T=  \left(u^{B+T} \prod_{\bar{a} \in T_k} (s^{\bar{a}} u)^{i_{\bar{a}}} \right) \prod_{\bar{a} \in T_{n-1} - T_k } (s^{\bar{a}} u)^{i_{\bar{a}} + T}  \in P_{1 < \cdots < k}^{E}  .$$ 
In particular, the monomial in parentheses is in $P_{1< \cdots < k}^E$ so it is a multiple of some $E$-fold product of generators of  $P_{1< \cdots < k} = (s^{\bar{a}} u \colon \bar{a} \in T_k)V_D$. Thus we infer that two inequalities must hold, signifying we have enough $u$'s and $s_j$'s ($1 \le j \le k$) at our disposal, respectively, to feasibly form such a $E$-fold product. %of generators of  $P_{1< \cdots < k}$. 
These inequalities are \textbf{(1)} $\sum_{\bar{a}\in T_k} i_{\bar{a}} + B+ T \ge E$,  %\footnote{This inequality means we have enough $u$'s showing up to re-write $g$ to involve an $N$-fold product of generators of $P_{1 < \cdots < k}$.}
 and \textbf{(2)} the sum $$\sum_{\bar{a} \in T_k} i_{\bar{a}} (a_1 + \cdots + a_k) = \sum_{j=1}^D \ell_j \cdot j \ge E,$$ where $\ell_j := \sum_{\bar{a} \in T_{k, j}} i_{\bar{a}}$, 
$T_{k,j} := \{ \bar{a} \in T_k \colon \mbox{ the partition }a_1 + \cdots + a_k = j \}$. 
Indeed, $$E  \le \sum_{j=1}^D \ell_j \cdot j \le D \left(\sum_{j=1}^D \ell_j \right) \Longrightarrow \sum_{j=1}^D \ell_j  \ge  \lceil E/D \rceil , $$ so \textbf{(2)} implies that  \textbf{(3)} 
$\sum_{\bar{a} \in T_k} i_{\bar{a}}  = \sum_{j=1}^D \ell_j \ge \lceil E/D \rceil$.\footnote{Together, inequalities \textbf{(1)} and \textbf{(3)} are equivalent to 
$$\sum_{\bar{a} \in T_k} i_{\bar{a}}  = \sum_{j=1}^D \ell_j \ge \max \{\lceil E/D \rceil  , E - (B+T)\} = \lceil E/D \rceil \mbox{ for all }T \ge E.$$} 
For any monomial $g \in P_{1< \cdots < k}^{(E)}$, \textbf{(3)} implies that $g \in P_{1< \cdots < k}^{\lceil E/D \rceil}$. Thus  $P_{1< \cdots < k}^{(E)} \subseteq P_{1< \cdots < k}^{\lceil E/D \rceil}$ for all $E > 0$. 

Additionally if we consider $R$ with its \textbf{standard $\N$-grading}, then the minimal degree of a monomial (e.g., a monomial generator) in $P_{1< \cdots < k}^r$ is $r$. Noticing that for $1 \le j \le k $, the degree 
$\lceil E/D \rceil$ monomial $(s_j^D u)^{\lceil E/D \rceil} \in P_{1< \cdots < k}^E : (u^{(E +1) - \lceil E/D \rceil}) \subseteq  P_{1< \cdots < k}^E : (m^{(E +1) - \lceil E/D \rceil}) \subseteq  P_{1< \cdots < k}^{(E)}$, we obtain the only-if part of: for each $1 \le  k \le n$, $P_{1< \cdots < k}^{(E)} \subseteq P_{1< \cdots < k}^r \mbox{ if and only if }r \le \lceil E/D \rceil .$  

Setting $E = Dr-(D-1) = D(r-1)+1$, we have $\lceil E/D  = (r-1) + 1/D \rceil = r$, so that $P_{1< \cdots < k}^{(Dr-(D-1))} \subseteq P_{1< \cdots < k}^r$ for all $r>0$ and this containment is sharp. 

In review, our argument \textbf{does not} depend crucially on which size-$k$ index subset $j_1 < \ldots < j_k$ of $[n] = 
\{1, 2, \ldots, n\}$ we worked with; going with $1 < 2 < \ldots < k$ merely simplifies %some of the
 notation. In other words, in applying suitable permutations of the algebra generators for $V_D$, one obtains the above characterization of ideal containment for all of the %distinguished graded 
monomial prime ideals in the ring \textbf{having one of the $P_j$ as an ideal summand.} %, barring the zero ideal and the graded maximal ideal. 
To handle monomial primes having the height one prime 
$$P_{(-1, \ldots, -1, D)} = \left(s_1^{a_1} \cdots s_{n-1}^{a_{n-1}} u \colon 0 \le \sum_{i=1}^{n-1} a_i  \le D-1 \right)$$ as a summand, we use the $\F$-algebra isomorphisms %bijection of sets 
 $\phi_j \colon V_D \to V_D$ ($1 \le j \le n-1$) under which a monomial algebra generator
$g = s_1^{a_1} \cdots s_j^{a_j} \cdots s_{n-1}^{a_{n-1}} u$ with $0 \le A: = \sum_{i=1}^{n-1} a_i  \le D$ is sent to 
\begin{align*}
\phi_j (g) = \begin{cases}s_1^{a_1} \cdots s_j^{D-A} \cdots s_{n-1}^{a_{n-1}} u &\mbox{ if }A \le D-1 \mbox{ and } a_j = 0\\ 
 s_1^{a_1} \cdots s_j^{0} \cdots s_{n-1}^{a_{n-1}} u &\mbox{ if }A=D \mbox{ and } a_j > 0\\ 
 g &\mbox{ if }A \le D-1 \mbox{ and } a_j > 0 \\ 
 g &\mbox{ if }A=D \mbox{ and } a_j = 0. \end{cases}
\end{align*}
%$s_1^{a_1} \cdots s_{n-1}^{a_{n-1}} (s_j u)$. that swap $u$ and $s_j u$ while acting as the identity on the other algebra generators: 
We note that $\phi_j^2 = \phi_j \circ \phi_j$ is the identity, and the height one prime $\phi_j (P_{(-1, \ldots, -1, D)}) = P_j$: indeed, when $h=  s_1^{a_1} \cdots s_j^{a_j} \cdots s_{n-1}^{a_{n-1}} u$ is a generator of $P_j$, $a_j > 0$; when $A \le D-1$, $h = \phi_j (h)$, or else $D-A = 0$,   
$a_j = D - \left(\sum_{1 \le i \neq j \le n-1} a_i \right) > 0 $, and $h = \phi_j (g) $ where $g = s_1^{a_1} \cdots s_j^{0} \cdots s_{n-1}^{a_{n-1}} u \in P_{(-1, \ldots, -1, D)}$. 
 %This shows $\phi_j (P_{(-1, \ldots, -1, D)})  \subseteq P_j $ and the other containment follows from the $A \le D-1$ case of our rule since $a_j +1 \ge 1$. 
Moreover, %Applying $\phi_j^{-1}$, 
 we conclude that a (sharp) containment $Q^{(m)} \subset Q^r$ for any monomial prime $Q$ with $P_j$ as a summand translates under $\phi_j$ to a (sharp) containment $(Q')^{(m)} \subset (Q')^r$ for a monomial prime $Q'$ of the same height as $Q$, with $P_{(-1, \ldots, -1, D)}$ replacing $P_j$ as an ideal summand. Having analyzed ideals with one of the  $P_j$ as a summand quite thoroughly, this final observation completes the proof. \qed 
\end{proof}

As advertised in the introduction, we want to close by drawing a connection between Lemma \ref{thm: du Val bound 1} and Theorems \ref{thm: Dth power hypersurface 1} and \ref{thm: Veronese rings are optimal 1}, e.g., to see that the containments in the lemma can be tight by example.

\begin{remark}\label{rem:class-group-presentation}
With notation as in Theorem \ref{thm: exact sequence 00}, we note that if $C \subseteq N_\RR$ is a full pointed rational polyhedral cone, then we have the following presentation for the divisor class group: 
$$\operatorname{Cl}(\F[C^\vee \cap M]) \cong 
\frac{\bigoplus_{\rho \in \Sigma(1)} \Z \cdot [D_\rho]  }{\langle \sum_{\rho \in \Sigma(1)}  \langle e_i^* , u_\rho\rangle  [D_\rho] =0 \colon 1 \le i \le n  \rangle },$$
where the $e_i^* \in M$ form the dual basis to the basis $e_1, \ldots, e_n \in N$ chosen in $N$. 
\end{remark}

\begin{exm}\label{exm:toric-class-group-computations} We work with the polyhedral cones in the proof of Theorem \ref{thm: Veronese hypersurf}, showing that $\mbox{Cl}(H_D) \cong (\Z/D\Z)^{n-1}$ and $\mbox{Cl}(V_D) \cong \Z / D\Z$. Although these class group facts are well known in certain circles and can be deduced by other means $($see e.g., \cite{SinghSpiroff000}$)$, for completeness of exposition we include  succinct computations. % using toric divisor theory. %, and we adduce this connection with the class groups in suggesting our Harbourne-Huneke type bounds have a group-theoretic flavor. . 
\begin{enumerate}
\item The cone $\sigma_{D}^{(n)} \subseteq N_\RR$ has ray generators  $f_i = D e_i + e_n$ for $1 \le i < n$ and 
$e_n$, and 
\begin{align*}\mbox{Cl}(\F[(\sigma_{D}^{(n)})^\vee \cap \Z^n ]) &\cong 
\frac{ \Z \cdot \mathbf{[D_{e_n}]} \oplus \bigoplus_{i=1}^{n-1} \Z \cdot [D_{f_i}] }{\langle  D [D_{f_i}]  = 0 \mbox{ }(1 \le i < n),     \mathbf{[D_{e_n}] =  -  [D_{f_1}] - \cdots - [D_{f_{n-1}}] } \rangle} \\
&\cong 
\frac{\Z \cdot \mathbf{ - [D_{f_1}] - \cdots - [D_{f_{n-1}}]} \oplus \bigoplus_{i=1}^{n-1} \Z \cdot [D_{f_i}] }{\langle  D [D_{f_1}]  = 0 , \ldots, D [D_{f_{n-1}}]  = 0   \rangle} \\
&=  \frac{\bigoplus_{i=1}^{n-1} \Z \cdot [D_{f_i}] }{\langle  D [D_{f_1}]  = 0 , \ldots, D [D_{f_{n-1}}]  = 0   \rangle} \\ &\cong \boxed{(\Z/ D \Z)^{n-1} .} 
\end{align*}
\item  The cone $\eta_{D}^{(n)} \subseteq N_\RR$ has ray generators  $e_i$ for $1 \le i < n$ and 
$f_n = D e_n - \sum_{i=1}^{n-1} e_i$, and 
\begin{align*}\mbox{Cl}(\F[(\eta_{D}^{(n)})^\vee \cap \Z^n ]) &\cong 
\frac{\Z \cdot \mathbf{[D_{f_n}]} \oplus \bigoplus_{i=1}^{n-1} \Z \cdot [D_{e_i}] }{\langle  \mathbf{[D_{e_i}] -  [D_{f_n}] = 0 \mbox{ }(1 \le i < n)}, \mbox{ } D  [D_{f_n}] = 0 \rangle} \\ 
&\cong 
\frac{\Z \cdot  [D_{f_n}]  }{\langle  D [D_{f_n}]  = 0 \rangle} \\ &\cong \boxed{(\Z/ D \Z) .} 
\end{align*}
\end{enumerate}
\end{exm}

%\newpage

\section{Lingering Questions related to Theorem 1.2}\label{section: Finale 1}

To summarize, we have deduced two existence criteria for uniform Harbourne-Huneke bounds.  Lemma \ref{thm: du Val bound 1} holds for ideals of pure height one in a Noetherian normal domain. %; it can be applied to familiar classes of local- or graded local Cohen-Macaulay domains with 
%pseudorational
%rational singularities (see Section 3 of (\cite{Walker001})).
And Theorem \ref{thm: finite tensor products 000} holds for monomial primes in finite tensor products of normal toric rings; %, where the factors themselves satisfy Harbourne-Huneke bounds on monomial primes; 
we deduced Theorem \ref{thm: Veronese hypersurf} to increase the range of examples that can be used as tensor factors. These criteria cover a prodigious class of normal toric rings.  %We have also taken first steps towards witnessing Harbourne-Huneke bounds on monomial prime ideals in simplicial toric rings of dimension at least three. 
%Aside from strengthening Lemma \eqref{thm: du Val bound 1} to cover all height one ideals, 
 We include the following illustrative example: 
\begin{exm}
%[A  ``Zany" Example from the Primorial Ooze (A Pun!)] 
Let $p_i \in \mathbb{N}$ be the $i$-th prime number, $R_i$  the $p_i$-th Veronese subring of $\mathbb{F}[X_{i , 1}, \ldots, X_{i , 14641}]$. Set $$R(n) = (\bigotimes_{i=1}^n)_{\mathbb{F}} R_i, \quad \sigma(n)= \prod_{i=1}^n p_i \mbox{ $($the \textbf{primorial function}$)$}.$$ 
One can compute that $\operatorname{Cl}(R(n)) \cong \mathbb{Z} / \sigma(n)$ via toric divisor theory, so Lemma \ref{thm: du Val bound 1} says that 
$$\mathfrak{q}^{(\sigma(n) (r-1) + 1)} \subseteq \mathfrak{q}^r$$  for all ideals $\mathfrak{q} \subseteq R(n)$ of pure height one, and all $r> 0$. Also, $D = p_n$ in Theorem \ref{thm: finite tensor products 000}, covering \textbf{all $2^{14641 n}$ monomial primes in $R(n)$}. The multiplier for monomial primes climbs much slower than the multiplier in pure height one as $n$ climbs to infinity.  
\end{exm}
We close %by briefly mentioning
 with a few natural lines for further investigation. 

\begin{enumerate}
%\item Can the conclusion of Lemma \eqref{thm: du Val bound 1} be strengthened to cover all ideals of height one?  
%However, we originally witnessed these inclusions at the level of symbolic powers of primes by studying simplicial toric rings in Macaulay2 (\cite{M2}), working over perfect fields such as $\Q$ and $\Z/p \Z$ ($p$ prime) that are not algebraically closed, and, more to the point, that do not generate error messages when using functions implemented in Macaulay2. 

\item Does the conclusion of Theorem \ref{thm: finite tensor products 000} extend to monomial primes in any simplicial toric ring? Can we identify a candidate mechanism (e.g., group-theoretic) to help explain and verify these Harbourne-Huneke bounds in height two or higher for a larger class of ideals than monomial primes? 

\item Given the role of tensor products in our manuscript, do analogues of Theorems \ref{thm: finite tensor products 000} and \ref{thm: finite tensor products 001} hold for other graded ring constructions in the toric setting, such as Segre products?

%\item As noted in section 4 of (\cite{SzemSzpo01}), several of the known characteristic zero counterexamples to the $I^{(3)} \subseteq I^2$ containment have a reflection group of finite order   %of rigid motions (e.g., symmetries)
% associated to them. It might be worth investigating whether a group-theoretic analogue of Harbourne-Huneke bounds is salvageable for these counterexamples.    
\end{enumerate}

\bibliographystyle{amsplain}

\bibliography{USTPFlatExt}

\providecommand{\bysame}{\leavevmode\hbox to3em{\hrulefill}\thinspace}
\providecommand{\MR}{\relax\ifhmode\unskip\space\fi MR }
\providecommand{\MRhref}[2]{%
  \href{http://www.ams.org/mathscinet-getitem?mr=#1}{#2}
}
\providecommand{\href}[2]{#2}
\begin{thebibliography}{10}

\bibitem{Akes001}
S.~Akesseh, \emph{{Ideal Containments Under Flat Extensions}}, J. Algebra
  \textbf{492} (2017), 44--51.

\bibitem{Alt-Klei13}
A.~Altman and S.~Kleiman, \emph{{A Term of Commutative Algebra}}, Worldwide
  Center of Mathematics LLC, Cambridge, MA, 2014.

\bibitem{Primer}
T.~Bauer, S.~Di Rocco, B.~Harbourne, M.~Kapustka, A.L. Knutsen, W.~Syzdek, and
  T.~Szemberg, \emph{{A primer on Seshadri constants}}, Contemporary
  Mathematics \textbf{496} (2009), pp. 33-70
  \href{https://arxiv.org/abs/0810.0728}{arXiv/0810.0728}.

\bibitem{BH1}
{C. Bocci and B. Harbourne}, \emph{Comparing powers and symbolic powers of
  ideals}, J. Algebraic Geom. \textbf{19} (2010), no.3, pp. 399-417.

\bibitem{torictome}
D.A. Cox, J.B. Little, and H.K. Schenck, \emph{{Toric Varieties, Graduate
  Studies in Mathematics 124}}, American Mathematical Society, Providence, RI,
  2011.

\bibitem{5authorSymbolicSurvey}
H.~Dao, A.~De Stefani, E.~Grifo, C.~Huneke, and L.~N\'{u}{\~{n}}ez-Betancourt,
  \emph{{\em Symbolic Powers of Ideals}}, To appear in \textit{Advances in
  Singularities and Foliations: Geometry, Topology and Applications}, Springer
  Proceedings in Mathematics \& Statistics.
  \href{https://arxiv.org/abs/1708.03010}{arXiv/1708.03010}, 2017.

\bibitem{DSTG01}
M.~Dumnicki, T.~Szemberg, and H.~Tutaj-Gasi\'{n}ska, \emph{{Counterexamples to
  the $I^{(3)} \subseteq I^2$ containment}}, J. Algebra \textbf{393} (2013),
  pp.24-29. \href{http://arxiv.org/abs/1301.7440}{arXiv/1301.7440}.

\bibitem{ELS}
L.~Ein, R.~Lazarsfeld, and K.~Smith, \emph{{{\em \textit{Uniform bounds and
  symbolic powers on smooth varieties}}}}, Invent. Math. \textbf{144} (2001),
  pp. 241-252. \href{https://arxiv.org/abs/math/0005098}{arXiv/0005098}.

\bibitem{fossum2012divisor}
R.M. Fossum, \emph{The divisor class group of a krull domain}, vol.~74,
  Springer Science \& Business Media, 2012.

\bibitem{introtoric}
W.~Fulton, \emph{{Introduction to Toric Varieties, Annals of Math. Studies
  131}}, Princeton University Press, Princeton, NJ, 1993.

\bibitem{M2}
D.R. Grayson and M.E. Stillman, \emph{{Macaulay 2, a software system for
  research in algebraic geometry}}, Available at
  http://www.uiuc.edu/Macaulay2/, 1992.

\bibitem{GrifoHun00}
E.~Grifo and C.~Huneke, \emph{{\em \textit{Symbolic powers of ideals defining
  F-pure and strongly F-regular rings}}}, International Mathematics Research
  Notices, DOI:10.1093/imrn/rnx213
  \href{https://arxiv.org/abs/1702.06876}{arXiv/1702.06876}, 2017.

\bibitem{resurge2}
B.~Harbourne and A.~Seceleanu, \emph{{Containment counterexamples for ideals of
  various configurations of points in $\mathbb{P}^n$}}, J. Pure Appl. Algebra
  \textbf{219} (2015), no.4, pp. 1062-1072.
  \href{http://arxiv.org/abs/1306.3668}{arXiv/1306.3668}.

\bibitem{Hartsh0}
R.~Hartshorne, \emph{{Algebraic Geometry, Graduate Texts in Math.
  \textbf{52}}}, Springer-Verlag, New York, 1977.

\bibitem{hoch2}
M.~Hochster, \emph{{Math 615 Winter 2007 Lecture 4/6/07}}, Available at
  \href{http://www.math.lsa.umich.edu/~hochster/615W07/L04.06.pdf}{Online
  link.}, 2007.

\bibitem{HH1}
M.~Hochster and C.~Huneke, \emph{{Comparison of ordinary and symbolic powers of
  ideals}}, Invent. Math. \textbf{147} (2002), pp. 349-369
  \href{https://arxiv.org/abs/math/0211174}{arXiv/0211174}.

\bibitem{Lip0}
J.~Lipman, \emph{{Rational singularities, with applications to algebraic
  surfaces and unique factorization}}, Inst. Hautes \'{E}tudes Sci. Publ. Math.
  \textbf{36} (1969), pp. 195--279.

\bibitem{Matsumura}
H.~Matsumura, \emph{{Commutative Ring Theory}}, Cambridge Univ. Press,
  Cambridge, MA, 1989.

\bibitem{SinghSpiroff000}
A.K. Singh and S.~Spiroff, \emph{Divisor class groups of graded hypersurfaces},
  Contemporary Mathematics \textbf{448} (2007), pp. 237-243.

\bibitem{SzemSzpo01}
T.~Szemberg and J.~Szpond, \emph{On the containment problem},
  \href{http://arxiv.org/abs/1601.01308}{arXiv/1601.01308}, 2016.

\bibitem{walker2016}
R.M. Walker, \emph{Rational singularities and uniform symbolic topologies},
  Illinois J. Math. \textbf{60} (2016), no.~2, 541--550.

\bibitem{Walker003}
R.M. Walker, \emph{{\em Uniform Symbolic Topologies via Multinomial
  Expansions}}, To appear in \textit{Proceedings of the AMS}.
  \href{https://arxiv.org/abs/1703.04530}{arXiv/1703.04530}, 2017.

\end{thebibliography}

\end{document}